\documentclass{article}
\usepackage[left=1in,right=1in,top=1in,bottom=1in]{geometry}

\usepackage{cite}
\usepackage{mathtools}
\usepackage{amssymb}
\usepackage{algorithmicx}
\usepackage{algpseudocode}
\usepackage{hyperref}
\hypersetup{
pdftitle={Relaxed multibang regularization for the combinatorial
integral approximation},
pdfauthor={Paul Manns},
}
\usepackage{multirow}
\usepackage{xcolor}
\usepackage{pgfplots}
\usepackage{tikz}
\usepackage{doi}
\usepackage{subcaption}
\usepackage{siunitx}
\usepackage[nolist]{acronym}
\usepackage{bbm}
\usepackage{algorithm}
\usepackage{algorithmicx}

\usepackage{amsthm}

\newtheorem{theorem}{Theorem}[section]
\newtheorem{proposition}[theorem]{Proposition}
\newtheorem{lemma}[theorem]{Lemma}
\newtheorem{corollary}[theorem]{Corollary}

\theoremstyle{definition}
\newtheorem{definition}[theorem]{Definition}

\newtheorem{example}[theorem]{Example}

\theoremstyle{remark}
\newtheorem{remark}[theorem]{Remark}

\numberwithin{equation}{section}

\DeclareMathOperator*{\argmin}{arg\,min}
\DeclareMathOperator*{\conv}{conv}
\DeclareMathOperator*{\ext}{ext}
\DeclareMathOperator*{\epi}{epi}

\DeclareMathOperator*{\interior}{int}

\newcommand*\dd{\mathop{}\!\mathrm{d}}

\newcommand{\weakstarto}{\stackrel{\ast}{\rightharpoonup}}
\newcommand{\RG}{\mathcal{T}}
\DeclareMathOperator*{\RA}{\mathcal{RA}}

\algdef{SE}[DOWHILE]{Do}{doWhile}{\algorithmicdo}[1]{\algorithmicwhile\ #1}%

\usepackage[nolist]{acronym}

\newcommand{\R}{\mathbb{R}}
\newcommand{\N}{\mathbb{N}}

\definecolor{darkgreen}{rgb}{0,0.5,0}

\title{Relaxed multibang regularization for the combinatorial
integral approximation\thanks{%
	Submitted to the editors DATE}}
\author{Paul Manns
\thanks{Mathematics and Computer Science Division, 
Argonne National Laboratory, Lemont, IL 60439, U.S.A.
        (\texttt{pmanns@anl.gov})}}

\begin{document}
\maketitle
\begin{abstract}
Multibang regularization and combinatorial integral
approximation decompositions are two actively researched
techniques for integer optimal control.
We consider a class of polyhedral functions that
arise particularly as convex lower envelopes of multibang regularizers
and show that they have beneficial properties with respect to
regularization of relaxations of integer optimal control problems.
We extend the algorithmic framework of the
combinatorial integral approximation such that a subsequence
of the computed discrete-valued controls converges to the
infimum of the regularized integer control problem.
\end{abstract}

\section{Introduction}

We consider the following class of integer optimal control problems:
\begin{gather}\label{eq:p}
\begin{aligned}
\inf_v\ & F(v) + R(v) \\
\text{ s.t.\ }& v \in L^\infty(\Omega,\R^m)
\text{ and }
v(x) \in \{\nu_1,\ldots,\nu_M\} \eqqcolon V \text{ for almost all (a.a.) } x\in \Omega.
\end{aligned}\tag{P}
\end{gather}
Here, $\Omega \subset \R^d$ for $d \in \N$
is a bounded domain. The optimized function $v$
is called the control input of the problem and may 
attain only values in the set of \emph{bangs}
$V \subset \R^m$ that has finite cardinality $|V| = M$.
The function $R : L^\infty(\Omega, \R^m) \to \R$
is a regularizer for the control input
and is of the form $R(v) = \int_\Omega g(v(x))\dd x$
for a proper convex lower semicontinuous
function $g : \conv V \to \R$.
The function $F : L^\infty(\Omega,\R^m) \to \mathbb{R}$
is convex and maps
weakly-$^*$-convergent sequences in $L^\infty(\Omega,\R^m)$
to convergent sequences in $\R$ (weakly-$^*$-sequentially 
continuous function).

Apart from the discreteness constraint
$v(x) \in \{\nu_1,\ldots,\nu_M\}$, this setting is typical
for PDE-constrained optimization, and a usual choice for $F$
is the composition of a convex function with the solution
operator of some initial or boundary value problem.
We relax the constraint $v(x) \in \{\nu_1,\ldots,\nu_M\}$ to
$v(x) \in \conv\{\nu_1,\ldots,\nu_M\}$, where $\conv$ denotes the
convex hull operator, and obtain the continuous
relaxation of \eqref{eq:p}:
\begin{gather}\label{eq:r}
\begin{aligned}
\min_v\ & F(v) + R(v) \\
\text{ s.t.\ }& v \in L^\infty(\Omega,\R^m)
\text{ and }
v(x) \in \conv\{\nu_1,\ldots,\nu_M\} \text{ for a.a.\ } x\in \Omega.
\end{aligned}\tag{R}
\end{gather}

The combinatorial integral approximation \cite{sager2011combinatorial,jung2015the}
decomposes the solution process of \eqref{eq:p} into two steps.
\begin{enumerate}
\item Solve the continuous relaxation \eqref{eq:r} of \eqref{eq:p}.
\item Solve an approximation problem (also called rounding problem)
to approximate the solution (control) computed in the first step in the weak-$^*$-topology.
\end{enumerate}

We use $\min$ in the definition of \eqref{eq:r} because we
generally seek for or assume settings such that
\eqref{eq:r} admits a minimizer.
Let $\inf \eqref{eq:p}$ denote the infimal value of \eqref{eq:p}
and $\min \eqref{eq:r}$ denote the minimal value of \eqref{eq:r}.
If the identity
\begin{gather}\label{eq:ciaidentity}
\min \eqref{eq:r}
= \inf \eqref{eq:p}
\end{gather}
holds, $\inf \eqref{eq:p}$ can be approximated
arbitrarily close with the combinatorial integral approximation
decomposition, for example, by following
the algorithmic framework in \cite{manns2020multidimensional}.

Integer control problems have a wide range
of applications from topology optimization \cite{haslinger2015topology}
to \emph{filtered approximation} in electronics \cite{buchheim2012effective}.
The focus of this article is on
the regularizer $R$ and the relationship between
\eqref{eq:p} and \eqref{eq:r}. Regularizers in integer optimal control problems may be used to account for economic costs incurred by different modes of
operation; see, e.g.\ \cite{sager2013sampling}, where measurement costs
and information gain are related using $L^1$-penalization.
Regularization terms have also been
suggested to promote structural properties
like sparsity or discreteness of relaxed integer controls in order to
facilitate the solution process; see 
\cite{leyffer2021convergence,clason2016convex,borrvall2001topology}
for topology optimization, 
\cite{garmatter2019improved,sharma2020inversion}
for source location problems,
and \cite{stadler2009elliptic} for
actuator location identification.

Recent research on the combinatorial integral approximation
has focused on the types of dynamical systems or problem settings
for which
the required properties of $F$ hold
\cite{yu2019multidimensional,manns2020multidimensional,manns2020improved,kirches2020compactness},
as well as algorithmic improvements for the second step
\cite{hante2013relaxation,manns2020multidimensional,
zeile2018combinatorial,bestehorn2019switching,bestehorn2020mixed,jung2015the}.
All of these articles analyze the setting $R = 0$.
A regularization of \eqref{eq:p} and \eqref{eq:r} is
either not considered or included only in the second step.
The reason is that many common choices for regularizers, 
such as $R(v) = \int_\Omega \|v(x)\|^2_2\dd x$,
exhibit strict convexity and thus are generally
incompatible with the identity \eqref{eq:ciaidentity}.
In fact, the following proposition holds.

\begin{proposition}\label{prp:counter_statement_strict_convex}
Let $g: \conv V \to [0,\infty)$ be strictly
convex and continuous.
Let $R(v) \coloneqq \int_{\Omega} g(v(x))\,\dd x$.
Let $\bar{v} \in L^\infty(\Omega,\R^{m})$ be $\conv V$-valued.
Let $A \subset \Omega$ be a set of strictly positive measure
with $\bar{v}(x) \notin V$ for a.a.\ $x \in A$.
Let $(v^n)_n \subset L^\infty(\Omega,\R^m)$ satisfy
$v^n(x) \in V$ for a.a.\ $x \in \Omega$ and all $n \in \N$,
and $v^n \weakstarto \bar{v}$. Then, $R(\bar{v}) < \liminf R(v^n)$.
\end{proposition}
\begin{proof}
The proof of Proposition \ref{prp:counter_statement_strict_convex}
is deferred to Appendix \ref{sec:proof_prp_counter_strict_cvx}.
\end{proof}

This implies that if $R$ is induced by a strictly
convex function $g$, we have
\[ \inf \eqref{eq:p} > \min \eqref{eq:r} \]
unless a solution of \eqref{eq:r} is already discrete-valued.
This is closely related to the fact that
for solutions $v$ of \eqref{eq:r} that are not
$V$-valued a.e., we cannot expect $v^n \to v$
if the $v^n$ are $V$-valued even if \eqref{eq:ciaidentity}
holds and $v^n \weakstarto v$; see also 
\cite[Cor.\,11]{clason2018vector}.

Therefore, we strive for a class of
functions $R$ such that the identity \eqref{eq:ciaidentity}
still holds true and the combinatorial integral approximation
framework is applicable to \eqref{eq:p}.
We build on the ideas and analysis presented in
\cite{clason2018vector} and consider functions $R$
of the form $R(v) = \int_\Omega g(v(x))\dd x$,
where $g$ is not strictly convex but has a
polyhedral epigraph instead.
Such functions are relaxations of multibang regularizers
and arise as their convex lower envelopes; see
\cite{clason2014multi,clason2016convex,clason2018vector}.
If the set $\{ (\nu_i, g(\nu_i)) \,|\, i \in \{1,\ldots,M\} \}$
is the set of vertices (extremal points)
of the epigraph of $g$, the identity \eqref{eq:ciaidentity}
still holds.

To use these regularizers
in the combinatorial integral approximation,
we extend the algorithmic framework from \cite{manns2020multidimensional}.
In \cite{manns2020multidimensional}, the identity \eqref{eq:ciaidentity}
and convergence of the algorithmic
framework are shown for the case that $R = 0$ and $F$ is the
composition of a function that depends only on the state
vector of a PDE, for which a Lax--Milgram type statement
and a compact embedding holds, with the control-to-state operator
of the PDE.
For the analysis in this work, we restrict to convex
functions $F$, as for example arise from compositions of convex objective
terms that only depend on the state vector with compact
control-to-state 
operators, and add the class of nonsmooth convex regularizers $R$
described above to the problem.
We take care of the nonsmoothness of the regularizers 
in the algorithmic framework by means of Moreau envelopes, for which
we obtain $\Gamma$-convergence.
The extended algorithm produces a sequence of discrete-valued controls
that admits at least one weak-$^*$-cluster point.
All weakly-$^*$-convergent subsequences are minimizing sequences
of \eqref{eq:p}.

We structure the remainder of the article as follows.
In Section \ref{sec:rmbr} we introduce and analyze
the relaxed multibang regularization.
In Section \ref{sec:algorithm} we provide the
extended algorithmic framework 
and prove convergence.
In Section \ref{sec:computational_examples} we present two examples to validate
our analysis computationally. 

\paragraph{Notation}

For $n \in \N$, we define $[n] \coloneqq \{1,\ldots,n\}$.
Let $\Omega \subset \R^d$ be a bounded domain. For a subset
$A \subset \Omega$, we
denote its relative complement with respect to $\Omega$,
by $A^c \coloneqq \Omega\setminus A$.
The Lebesgue measure is denoted by the symbol $\lambda$.
For two subsets $A$, $B \subset \mathcal{V}$ of
a vector space $\mathcal{V}$, we define the
Minkowski sum
$A + B \coloneqq \{ a + b\,|\, a \in A, b \in B\}$.
We abbreviate the feasible set of the optimization problem
\eqref{eq:r} by $\mathcal{F}_{\eqref{eq:r}}$, specifically
$\mathcal{F}_{\eqref{eq:r}} \coloneqq
\{ v \in L^\infty(\Omega,\mathcal{V}) \,\vert\, v(x) \in \conv\{\nu_1,\ldots,\nu_M\} \text{ for a.a.\ } x \in \Omega \}$.
For a set $A$, we denote its binary-valued indicator function by
the symbol $\chi_A$.

\section{Relaxed Multibang Regularization}\label{sec:rmbr}

In this section, we introduce and analyze relaxed
multibang regularizers and the relationship between \eqref{eq:r} and \eqref{eq:p}.
First, we consider scalar-valued controls in Section \ref{sec:scalar}.
Second, we introduce relaxed multibang regularizers
for vector-valued controls in Section \ref{sec:vector}.
Their integrands are convex polyhedral functions
with bounded domain that are
characterized as minimum values of pointwise-defined \acp{LP}.
In Section \ref{sec:selection_functions} we analyze the measurability of the pointwise-defined
functions. We give a constructive proof
of the identity \eqref{eq:ciaidentity} in Section \ref{sec:proof_of_cia_identity}.
In Section \ref{sec:smoothing} we prove $\Gamma$-convergence
for smoothing the regularizers with Moreau envelopes.

\subsection{Scalar-Valued Controls}\label{sec:scalar}

Let $[\nu_1,\nu_M] \subset \R$ and $\nu_1 \le \ldots \le \nu_M$,
in particular $\conv V = [\nu_1,\nu_M]$. Let $\gamma > 0$.
The convex lower envelope of the multibang regularizer for a one-dimensional
real domain is given explicitly in \cite{clason2016convex} and is
\begin{align*}
	R(v) &= \int_\Omega g(v(x)) \dd x \text{ with }\\
	g(u) &\coloneqq \left\{
	\begin{aligned}
	\frac{\gamma}{2}((\nu_i + \nu_{i+1})u - \nu_i \nu_{i+1})
	& \text{ for } u \in [\nu_i, \nu_{i+1}] \text{ for } i = 1,\ldots,M-1,\\
	\infty & \text{ for } u \notin [\nu_1,\nu_M].
	\end{aligned}
	\right.
\end{align*}
Let $v^n(x) \in \{\nu_1,\ldots,\nu_M\}$ a.e.\ for all $n \in \N$, and let
$v^n \weakstarto v$ in $L^\infty(\Omega)$. Then,
\begin{gather}\label{eq:rmbr_scalar_derivation}
R(v^n) = \int_\Omega g(v^n(x))\dd x
= \sum_{i=1}^{M}
\int_{\Omega}
\chi_{A^n_i} g(\nu_i) \dd x 
\to 
\sum_{i=1}^{M} 
\int_{\Omega} \alpha_i(x) g(\nu_i) \dd x,
\end{gather}
where $A^n_i = \{ x\in\Omega\,|\,v^n(x) = \nu_i \}$.
The function  $\alpha : \Omega 
\to [0,1]^M$ is well defined by virtue of the Lyapunov convexity theorem
\cite{lyapunov1940completely,tartar1979compensated}
and satisfies $\sum_{i=1}^M \alpha_i(x) = 1$ a.e.
The uniqueness of the weak-$^*$-limit gives $\sum_{i=1}^M\alpha_i(x)\nu_i = v(x)$
a.e.

For a.a.\ $x \in \Omega$, there is $i \in [M-1]$
such that $v(x) \in [\nu_i,\nu_{i+1}]$.
Assume that $\alpha$ encodes the corresponding convex coefficients,
specifically,
\begin{gather}\label{eq:next_cvx_coeffs_condition}
 v(x) = \alpha_i(x) \nu_i + \alpha_{i+1}(x) \nu_{i+1}
   \text{ and } \alpha_i(x) + \alpha_{i+1}(x) = 1
\end{gather}
and $\alpha_{j}(x) = 0$ for $j \notin \{i, i+1\}$
for a.a.\ $x\in \Omega$ such that
$v(x) \in [\nu_i,\nu_{i+1}]$.

Then, we may insert the definition of $g$ and
obtain for a.a.\ $x \in \Omega$ that
\begin{align*}
\frac{2}{\gamma}(\alpha_i(x)g(\nu_i) + \alpha_{i+1}(x)g(\nu_{i+1}))
	&= \alpha_i(x)(\nu_i + \nu_{i+1})\nu_i - \alpha_i(x)\nu_i\nu_{i+1} \\
	  &\quad + \alpha_{i+1}(x)(\nu_i + \nu_{i+1})\nu_{i+1}
	  - \alpha_{i+1}(x)\nu_i\nu_{i+1} \\
	&= (\nu_i + \nu_{i+1})v(x) - \nu_i\nu_{i+1}
	= \frac{2}{\gamma}g(v(x)).
\end{align*}
Inserting this identity into \eqref{eq:rmbr_scalar_derivation}, we obtain
\begin{gather}\label{eq:desired_convergence}
R(v^n) \to R(v).
\end{gather}
Choosing the convex coefficients $\alpha_i(x)$ such that
$\alpha_i(x) > 0$  holds only for the neighboring bangs $\nu_i$
and $\nu_{i+1}$ of $v(x)$
is the key for the convergence \eqref{eq:desired_convergence}.
This is illustrated in 
Figure \ref{fig:neighboring_cvx_illustration}.
Convex combinations of neighboring bangs $\nu_i$ enable 
the evaluation of the regularizer to commute with the evaluation of the
convex combination of the $\nu_i$.
This is not the case if convex combinations of non-neighboring bangs are used,
as is indicated
by the dotted green line in Figure \ref{fig:neighboring_cvx_illustration}
that lies strictly above $g$ in $(\nu_1,\nu_3)$.
We summarize these considerations in Proposition
\ref{prp:mblce_conv} below.
\begin{figure}[ht]
  \centering
  \includegraphics[width=\textwidth,height=3.5cm]{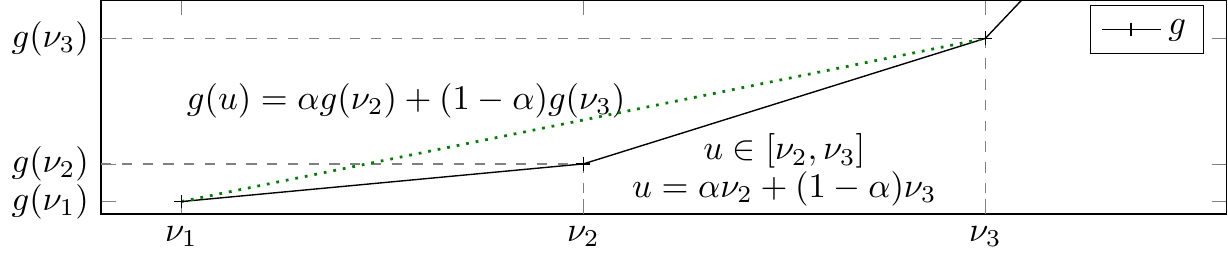}
%
%
%
\caption{Convex lower envelope of a multibang
regularizer and convex combinations
$u = \alpha \nu_2 + (1 - \alpha) \nu_3$.
For convex combinations between
$g(\nu_1)$ and $g(\nu_3)$ (dotted green line),
the evaluation of convex combinations
and the evaluation of $g$ do not commute if $u$ is represented
as a convex combination of $\nu_1$ and $\nu_3$.}
\label{fig:neighboring_cvx_illustration}
\end{figure}
\begin{proposition}\label{prp:mblce_conv}
Let $v^n \weakstarto v$ in $L^\infty(\Omega)$
with $v(x) \in [\nu_1,\nu_M]$ for a.a.\ $x\in\Omega$.
For a.a.\ $x \in \Omega$ assume that
$v(x) \in [\nu_i,\nu_{i+1}]$
for some $i \in [M-1]$ implies
$v^n(x) \in \{\nu_i,\nu_{i+1}\}$ for all $n \in \N$.
Then, $R(v^n) \to R(v)$.
\end{proposition}
\begin{proof}
The claim follows directly from the considerations above.
\end{proof}

Therefore, the function $R$ allows  the identity \eqref{eq:ciaidentity} to be preserved. However, this does not mean that $R$ is a weakly-$^*$-sequentially continuous function.

$R$ may be interpreted as a generalization
of the $L^1$-regularization
to promote $V$-valued controls with more than three bangs
\cite[Sect.\,3]{clason2016convex}.
The function $g$ above is the convex lower envelope
of the function
\begin{gather}\label{eq:mbreg}
 g_0(u) \coloneqq \frac{\gamma}{2}|u|^2
   + \beta\prod_{i=1}^M|u - \nu_i|^0 + \delta_{[\nu_1,\nu_M]},
\end{gather}   
where $\delta_{[\nu_1,\nu_M]}$ is the $\{0,\infty\}$-valued indicator
function of the set $[\nu_1,\nu_M]$, and $|0|^0 = 0$ and $|u|^0 = 1$ for
$u \neq 0$ if the ratio $\beta/\gamma$ is large enough;
see \cite[Sect.\,2]{clason2014multi} and \cite[Sect.\,3]{clason2016convex}.
Because we do not use this particular structure, the following statement follows.

\begin{corollary}\label{cor:pwc_lce_conv}
Let $g : [\nu_1,\nu_M] \to \R$ be a positive, piecewise affine
and convex function with the kinks connecting the affine pieces
at $\{\nu_1,\ldots,\nu_M\}$.
Let $R(v) \coloneqq \int_\Omega g(v(x))\dd x$.
Let $v^n \weakstarto v$ in $L^\infty(\Omega)$
with $v(x) \in [\nu_1,\nu_M]$ for a.a.\ $x\in\Omega$.
For a.a.\ $x \in \Omega$ assume that
$v(x) \in [\nu_i,\nu_{i+1}]$
for some $i \in \{1,\ldots,M-1\}$ implies
$v^n(x) \in \{\nu_i,\nu_{i+1}\}$ for all $n \in \N$.
Then, $R(v^n) \to R(v)$.
\end{corollary}
\begin{proof}
Since $g : [\nu_1,\nu_M] \to \R$ is a piecewise affine and convex function,
its Clarke subdifferential $\partial^c g$ is
\[ \partial^c g(u)
   = \left\{
   \begin{aligned}
   \{ L_i \} & \text{ if } u \in (\nu_{i},\nu_{i+1}) \text{ for some } i \in [M - 1],\\
   [-\infty,L_i] & \text{ if } u = \nu_1,\\
   [L_{i},L_{i+1}] & \text{ if } u = \nu_i \text{ for some } i \in [M-1],\\
   [L_M,\infty] & \text{ if } u = \nu_M\\
   \end{aligned}
   \right.
   \]
for some $-\infty < L_1 \le \ldots \le L_M < \infty$.
Thus, $g$ has a constant slope of $L_i$ on $(\nu_i,\nu_{i+1})$
for $i \in [M-1]$. Combining this with the
prerequisite that $v(x) \in [\nu_i,\nu_{i+1}]$
for some $i \in [M-1]$ implies
$v^n(x) \in \{\nu_i,\nu_{i+1}\}$ for all $n \in \N$, we again
obtain that the evaluation of
$g$ and the evaluation of the convex combination commute.
The rest of the proof follows with the arguments above. 
\end{proof}

\begin{remark}
The assumption that for a.a.\ $x \in \Omega$
the inclusion $v(x) \in [\nu_i,\nu_{i+1}]$
implies $v^n(x) \in \{\nu_i,\nu_{i+1}\}$ for all $n \in \N$
constrains the algorithms that compute the $V$-valued
controls $v^n$ for a given $v$ in the second step of the 
combinatorial integral approximation. Therefore, instead of computing the
$V$-valued controls $v^n$ directly, binary-valued approximations
$\omega^n$ of the coefficient function $\alpha$ are computed. Hence, the
convergence $R(v^n) \to R(v)$ hinges on satisfying
\eqref{eq:next_cvx_coeffs_condition} here.
This insight enters our analysis in the general definition of $g$
in Definition \ref{dfn:rel_mbreg} and the
proof of Lemma \ref{lem:rel_mbreg_Ridentity}.
\end{remark}

\subsection{Vector-Valued Controls}\label{sec:vector}

In \cite{clason2018vector}, the multibang regularizer
for vector-valued admissible controls
is introduced as the choice $g \coloneqq \alpha c + \delta_V$,
where $c$ is a positive strictly convex lower semicontinuous function
and $\delta_V$ denotes the $\{0,\infty\}$-valued indicator 
function of the set $V$. Relaxations of multibang regularizers
are then defined as the convex lower envelopes of such functions.

We take a different approach and define the considered class
of regularizers for multidimensional controls geometrically.
While this also yields polyhedral functions,
it directly fits the algorithmic framework of the
combinatorial integral approximation.

We recall that a function $f : \R^m \to \R \cup\{\infty\}$
is called \emph{polyhedral} if its epigraph
\[ \epi f = \left\{ (u,r) \in \R^m \times \R
                    \,|\, f(u) \le r \right\} \]
is a convex polyhedron. Next, we introduce the general class
of regularizers, on which we focus in the remainder of the article.

\begin{definition}\label{dfn:rel_mbreg}
Let $V \subset \R^m$, and let $\{g_1,\ldots,g_M\} \subset [0,\infty)$ satisfy
the identity
$\ext \left( \conv \{(\nu_i,g_i)\,|\,i\in[M]\} + \{ (0,r)\,|\, r \ge 0\} \right)
= \{ (\nu_i,g_i)\,|\, i \in [M]\}$.
Let $g : \R^m \to [0,\infty]$ be defined through
\[ 
g(u) \coloneqq \left\{
\begin{aligned}
\infty & \text{ if } u \notin \conv V,\\
\min \left\{ \sum_{i=1}^M \alpha_i g_i \,\Bigg|\,
\sum_{i=1}^M \alpha_i \nu_i = u, \sum_{i=1}^M \alpha_i = 1, \alpha \ge 0
\right\} & \text{ if } u \in \conv V.
\end{aligned}
\right.
\]
Then, we call the function $R(v) \coloneqq \int_\Omega g(v(x))\dd x$
a \emph{relaxed multibang regularizer}.
\end{definition}

The positive scalars $g_i$ may be interpreted as a means to encode a preference of the different discrete control 
values $\nu_i$. We establish well-definedness and basic properties of relaxed
multibang regularizers in the proposition below.

\begin{proposition}\label{prp:rel_mbreg_props}
\begin{enumerate}
\item\label{itm:gwelldefpoly} Let $g$ be as in Definition \ref{dfn:rel_mbreg}.
Then, $g$ is well defined and polyhedral.
\item\label{itm:glipschitz} Let $g$ be as in Definition
\ref{dfn:rel_mbreg}. Then, $g$ is Lipschitz continuous.
\item\label{itm:gpinduceschar} Let $g : \R^m \to [0,\infty]$ be polyhedral
with nonempty bounded domain.
Then, there exist $\{\nu_1,\ldots,\nu_M\} \subset \R^m$
and $\{g_1,\ldots,g_M\} \subset [0,\infty)$ such that
$g$ can be stated in the form of Definition \ref{dfn:rel_mbreg}.
\item\label{itm:Rcorrelmb} The function $R$ in Corollary \ref{cor:pwc_lce_conv}
is a relaxed multibang regularizer.
\end{enumerate}
\end{proposition}
\begin{proof}
\ref{itm:gwelldefpoly}. $u \in \conv V$ can be represented as a convex
combination of the $\nu_i$. 
Thus the feasible set
of the \ac{LP} defining $g(u)$ is nonempty.
Moreover, the feasible set is a polytope and the
\ac{LP} admits a minimizer. Consequently, $g(u)$ is well defined.

To show that $g$ is polyhedral, we prove that
$\epi g = \conv \{ (\nu_i,g_i)\,|\, i \in [M]\} + \{ (0, r) \,|\, r \ge 0 \}$.
The inclusion $\subset$ holds because $(u, g(u)) \in \conv \{ (\nu_i,g_i)\,|\, i \in [M]\}$ for $g(u) < \infty$.
To assert the inclusion $\supset$, we consider $(u,s)\in 
\conv \{ (\nu_i,g_i)\,|\, i \in [M]\}$ and $r \ge 0$.
Then $(u,s)$ can be represented
as a convex combination of the $(\nu_i,g_i)$, and the definition
of $g(u)$ gives $g(u) \le s \le s + r$. Thus $(u,s) + (0,r) \in \epi g$.

\ref{itm:glipschitz}. This follows from the Lipschitz
continuity of \acp{LP} with respect
to changes in the right-hand side in monotonic norms;
see \cite{mangasarian1987lipschitz}.

\ref{itm:gpinduceschar}. We use the inner description of the
convex polyhedron $\epi g$ and write $\epi g = P + C$, where $P$ is
a polytope and $C$
a finitely generated convex cone. Because the domain of $g$ is bounded
and $g \ge 0$, it follows that we may choose $C = \{ (0,r) \,|\, r \ge 0 \}$
and that $\epi g$ is pointed.
Thus, $\epi g$ has uniquely determined extremal points
such that $P = \conv \ext \epi g$.
Consequently, the claim follows by setting
$M \coloneqq |\ext \epi g|$ and
defining the $(\nu_i,g_i)$ as the extremal points of $\epi g$.

\ref{itm:Rcorrelmb}. Because the function $g$
in Corollary \ref{cor:pwc_lce_conv} is piecewise affine and convex, it
is polyhedral with domain $[\nu_1,\nu_M]$, which proves the claim.
\end{proof}

\begin{remark}
The assertions of Proposition \ref{prp:rel_mbreg_props}
still hold if the relaxed identity $\ext \conv \{ (\nu_i,g_i)\,|\, i \in [M]\}
   = \{ (\nu_i,g_i)\,|\, i \in [M]\}$
is used or is relaxed further to
$\interior \conv \{ (\nu_i,g_i)\,|\, i \in [M]\}
\cap \{ (\nu_i,g_i)\,|\, i \in [M]\} = \emptyset$.
\end{remark}

\begin{remark}\label{rem:uniqueness}
For $u \in \conv V$, we consider minimizers $\alpha(u)$ of $g(u)$, that is
\[ \alpha(u) \in \argmin \left\{
\sum_{i=1}^M \alpha_i g_i \,\Big|\,
u = \sum_{i=1}^M \alpha_i \nu_i, \sum_{i=1}^M \alpha_i = 1, \alpha \ge 0
\right\}.\]
The minimizers are unique in the setting of Corollary 
\ref{cor:pwc_lce_conv}, but this is not always true.
For example, consider
$V = \left\{(1\enskip 0)^T, (0\enskip 1)^T, (-1\enskip 0)^T, (0\enskip -1)^T \right\}$ and $g_i = 1$ for $i \in \{1,\ldots,4\}$. Then,
several convex combinations exist for $u = (0\enskip 0)^T$, and $\sum_{i=1}^4 \alpha_ig_i = 1$ holds for all of them.
\end{remark}

\subsection{Selection Functions for Convex Coefficients}\label{sec:selection_functions}

The algorithms for the second step of the combinatorial integral
approximation operate on functions of convex coefficients
$\alpha : \Omega \to [0,1]^M$ with $v(x) = \sum_{i=1}^M\alpha_i(x)\nu_i$
a.e.\ instead of the function $v$ directly.
Thus, for a given control $v : \Omega \to \conv V$,
we need to recover convex coefficients $\alpha(x) \in \R^M$ 
for a.a.\ $x \in \Omega$ that minimize the \acp{LP} defining $g(v(x))$.

We desire a measurable function $\alpha : \Omega \to [0,1]^M$.
Because the \acp{LP} defining $g(v(x))$ for $x \in \Omega$
do not necessarily have unique minimizers
and a selection function $\alpha(x)$ is not readily available in
closed form, the existence and computation of a measurable selection
function $\alpha$ are not immediate.
We consider the set-valued optimal policy
function $G : \conv V \rightrightarrows \R^M$, which is
defined as
\begin{gather}\label{eq:Gu}
G(u) \coloneqq \argmin \left\{
\sum_{i=1}^M \alpha_ig_i \,\Bigg|\, u = \sum_{i=1}^M \alpha_i \nu_i,
\sum_{i=1}^M \alpha_i = 1, \alpha \ge 0 \right\}
\end{gather}
for all $u \in \conv V$. Thus, we require a measurable
selector function $\tilde{g} : \conv V \to \R^M$
for $G$.
We recall that the multifunction $G$ is weakly measurable if
the set $\{ u \in \conv V\,|\, G(u) \cap U \neq\emptyset\}$
is Borel measurable for all open sets $U \subset \R^M$.
The following abstract result follows from the literature.

\begin{lemma}\label{lem:Gmeas}
The multifunction $G$ is weakly measurable and
admits a measurable selector $\phi : \conv V \to \R^M$.
\end{lemma}
\begin{proof}
We combine the Kuratowski--Ryll--Nardzewski measurable selection theorem
\cite{kuratowski1965general}
with the fact that the set $G(u)$ depends continuously
on the vector $u$ \cite{bohm1975continuity}.
\end{proof}

The measurability of the multifunction $G$ allows us to prove measurability
for a large class of possible selector functions.

\begin{lemma}\label{lem:gu_unique_welldefined}
Let $\eta : \R^M \to \R$ be strictly convex and lower semi-continuous.
Then, the function $\tilde{g}(u) \coloneqq \argmin\left\{ \eta(\alpha) \,|\, \alpha \in G(u) \right\}$
is single-valued and measurable.
In particular, $\tilde{g}$ defined as
$\tilde{g}(u) \coloneqq \argmin \{ \|\alpha\|_2^2 \,|\, \alpha \in G(u) \}$
for $u \in \conv V$ is measurable.
\end{lemma}
\begin{proof}
We observe that $G$ is weakly measurable (Lemma \ref{lem:Gmeas}) and compact-valued.
Then, we apply \cite[Thm.\,2\,\&\,3]{himmelberg1976optimal} with the
choices $u \coloneqq -\eta$ and $f \coloneqq \tilde{g}$. The optimization problem
$\min \{ \eta(\alpha) \,|\, \alpha \in G(u) \}$ has a unique solution for $u \in \conv V$.
The existence of a minimizer
follows from compactness of $G(u)$ and lower semi-continuity of $\eta$. The uniqueness follows from
the strict convexity of $\eta$. Therefore, the
selector $\tilde{g}$ is single-valued and can
be interpreted as a function $\tilde{g} : \conv V \to \R^M$.
\end{proof}

\subsection{Proof of the Identity \eqref{eq:ciaidentity}}\label{sec:proof_of_cia_identity}

We prove the identity \eqref{eq:ciaidentity}
for relaxed multibang regularizers by showing that for a
$\conv V$-valued control $v$, there exist $V$-valued
controls $v^n$ such that $R(v) = \liminf R(v^n)$ holds.
Before this result is proven in Lemma \ref{lem:rel_mbreg_Ridentity}, we
show an auxiliary lemma.

\begin{lemma}\label{lem:rel_mbreg_extremaloptimality}
Let $R(v) \coloneqq \int_\Omega g(v(x))\,\dd x$ be a relaxed
multibang regularizer. Then, the function $g$ defined
in Definition \ref{dfn:rel_mbreg} and $G$ defined in
\eqref{eq:Gu} satisfy $g(\nu_i) = g_i$ and $G(\nu_i) = \{ e_i \}$
for $i \in [M]$, where $e_i$ is the $i$th canonical unit vector in $\R^M$.
\end{lemma}
\begin{proof}
Let $i \in [M]$. The vector $e_i$ is feasible
for the \ac{LP} defining $g(\nu_i)$ with objective value $g_i$, 
which gives $g(\nu_i) \le g_i$.
Let $\alpha \in G(\nu_i)$ and $\alpha \neq e_i$.
We proceed by contradiction and distinguish the cases
$g(\nu_i) = g_i$ and $g(\nu_i) < g_i$.

Let $\sum_{i=1}^M \alpha_i \nu_i = \nu_i$ and
$\sum_{i=1}^M \alpha_i g_i = g_i$. Then,
\[ (\nu_i,g_i) \notin \ext \left( \conv \{(\nu_i,g_i)\,|\,i\in[M]\} + \{ (0,r)\,|\, r \ge 0\} \right), \]
a contradiction to Definition \ref{dfn:rel_mbreg}.
Let $\sum_{j=1}^M \alpha_j \nu_j = \nu_i$ and
$\sum_{j=1}^M \alpha_j g_j < g_i$. Then, we can write
$g_i = \eta \left(\sum_{j=1}^M \alpha_j g_j\right) + (1 - \eta) 2 g_i$
for some $\eta \in (0,1)$.
Clearly $\eta \sum_{j=1}^M\alpha_j \nu_j + (1 - \eta) \nu_i = \nu_i$.
Again, we obtain
\[ (\nu_i,g_i) \notin \ext \left( \conv \{(\nu_i,g_i)\,|\,i\in[M]\} + \{ (0,r)\,|\, r \ge 0\} \right), \]
which contradicts Definition \ref{dfn:rel_mbreg}.
Thus, $e_i$ is the only feasible point for the
\ac{LP} that defines $g(\nu_i)$, and we deduce
$g(\nu_i) = g_i$ and $G(\nu_i) = \{e_i\}$.
\end{proof}

\begin{lemma}\label{lem:rel_mbreg_Ridentity}
Let $R$ be a relaxed multibang regularizer.
Let $v \in L^\infty(\Omega, \R^m)$ with $R(v) < \infty$.
Then, there exists a sequence of functions
$(v^n)_n \subset L^\infty(\Omega, \R^m)$
that are $V$-valued such that
$v^n \weakstarto v$ in $L^\infty(\Omega, \R^m)$
and $R(v) = \lim_{n\to\infty} R(v^n)$.
\end{lemma}
\begin{proof}
Because $R(v) < \infty$, we can assume that
$v(x) \in \conv V$ a.e. Lemma
\ref{lem:Gmeas} gives the existence of
a measurable selector function $\phi$ for the LP solution sets that
defines the values of the integrand $g$ of $R$.
Because $v$ is measurable and $\conv V$-valued a.e.,
the function $\alpha : \Omega \to \R^M$
defined as $\alpha(x) \coloneqq \phi(v(x))$ a.e.\ is in $L^\infty(\Omega,\R^M)$.

Definition \ref{dfn:order_conserving_domain_dissection} implies
$\alpha(x) \ge 0$, $\sum_{i=1}^M \alpha_i(x) = 1$,
and $v(x) = \sum_{i=1}^M\alpha_i(x) \nu_i$.
We can apply the
multidimensional sum-up rounding algorithm
on a sequence of suitably refined grids
(see \cite{manns2020multidimensional}) to obtain a sequence
of $\{0,1\}^M$-valued functions $\omega^n$
such that $\sum_{i=1}^M \omega_i^n(x) = 1$ holds a.e.\
and for all $n$.
Then, $\omega^n \weakstarto \alpha$
follows from the analysis in \cite{manns2020multidimensional}.
We define $v^n(x) \coloneqq \sum_{i=1}^M \omega^n_i(x) \nu_i$ a.e.\ and for all $n$, and we deduce
\begin{align*}
R(v)
&= \int_\Omega g(v(x))\,\dd x
 = \int_\Omega \sum_{i=1}^M g_i \alpha_i(x)\,\dd x\\
&= \sum_{i=1}^M g_i \int_\Omega \alpha_i(x)\,\dd x
 = \lim_{n\to\infty} \sum_{i=1}^M g_i \int_\Omega \omega^n_i(x)\,\dd x
 = \lim_{n\to\infty} \int_\Omega \sum_{i=1}^M g_i \omega^n_i(x)\,\dd x.
\end{align*}
Thus, it remains to show that
\begin{gather}\label{eq:int_gomegan}
\int_\Omega \sum_{i=1}^M g_i \omega^n_i(x)\,\dd x 
   = \int_\Omega g\left(\sum_{i=1}^M \nu_i\omega^n_i(x)\right)\,\dd x
   = R(v^n).
\end{gather}   

For a.a.\ $x \in \Omega$, there exists
$i \in [M]$ such that $\omega^n_i(x) = 1$
and $\omega^n_j(x) = 0$ for $j \neq i$. This implies
\begin{gather}\label{eq:gomegan}
g(v^n(x))
  = g(\omega^n_i(x) \nu_i)
  = g(\nu_i) 
  \underset{\textrm{Lem.\ \ref{lem:rel_mbreg_extremaloptimality}}}= g_i
  = \sum_{i=1}^M g_i \omega^n_i(x).
\end{gather}  
Integrating \eqref{eq:gomegan} over $\Omega$
yields \eqref{eq:int_gomegan}, which closes the argument.
\end{proof}

Having obtained Lemma \ref{lem:rel_mbreg_Ridentity}, we are prepared
to prove the identity \eqref{eq:ciaidentity}.

\begin{theorem}
Let $\Omega$ be a bounded domain. Let $F : L^\infty(\Omega,\R^m) \to \R$ be weakly-$^*$-sequentially continuous. Let $V = \{\nu_1,\ldots,\nu_M\} \subset \R^m$
and $\{g_1,\ldots,g_M\} \subset [0,\infty)$
satisfy the assumptions of Definition \ref{dfn:rel_mbreg}, and let $R$
be defined as in Definition \ref{dfn:rel_mbreg}.
Then, \eqref{eq:r} admits a minimizer, and the identity \eqref{eq:ciaidentity}
holds.
\end{theorem}
\begin{proof}
Because $\R^m$ is a Euclidean space, the space $L^\infty(\Omega,\R^m)$ admits a weak-$^*$-topology.
From the Lyapunov convexity theorem---consider, for example, the version \cite[Thm.\,3]{tartar1979compensated}---we have that
the feasible set $\mathcal{F}_{\eqref{eq:r}}$ of \eqref{eq:r}
is weakly-$^*$-compact.

The weak-$^*$-sequential continuity of $F$ gives that $F$ is bounded on $\mathcal{F}_{\eqref{eq:r}}$. For $g$ defined in Definition \ref{dfn:rel_mbreg},
$\epi g$ is a convex polyhedron and therefore closed.
Thus, $g$ is a convex proper continuous function with bounded domain.
Consequently, $R : L^\infty(\Omega,\R^m) \to [0,\infty]$ is a weakly-$^*$-sequentially lower semicontinuous
function (see, e.g., \cite[Thm.\,5.14]{fonseca2007modern})
and bounded from below. The existence of a minimizer for \eqref{eq:r} follows with the direct method of calculus of variations.

Let $v$ be a minimizer of \eqref{eq:r}.
Lemma \ref{lem:rel_mbreg_Ridentity} gives that there exists a sequence of functions $(v^n)_n$
that are feasible for \eqref{eq:p} such that $R(v) = \lim R(v^n)$. The sequence $(v^n)_n$ provided by Lemma 
\ref{lem:rel_mbreg_Ridentity} satisfies $v^n \weakstarto v$.
The claim follows from the weak-$^*$-sequential continuity of $F$ and $R(v) = \lim R(v^n)$.
\end{proof}

\begin{remark}
The identity $\int_\Omega g(v(x))\,\dd x
 = \int_\Omega \sum_{i=1}^M g_i \alpha_i(x)\,\dd x$ in Lemma \ref{lem:rel_mbreg_Ridentity} is the
key step to prove the identity \eqref{eq:ciaidentity}. The
assumptions on the pairs $(\nu_i,g_i)$ and the pointwise
definition of $g$ as the minimum of an \ac{LP} allow
that evaluating convex combinations of the $\nu_i$ and evaluating $g$ commute.
Suboptimal convex combinations would result in a gap between $R(v)$
and $\liminf R(v^n)$; see again Figure \ref{fig:neighboring_cvx_illustration}.
\end{remark}

\subsection{Smoothing of $R$}\label{sec:smoothing}

We recall that the Moreau envelope of
a proper lower semicontinuous function
$f : \R^m \to \R \cup \{\infty\}$ is defined as
\[ (e_\gamma f)(x) \coloneqq
   \inf\left\{ f(y) + \frac{1}{2\gamma}\|x - y\|^2_2 \,\Big|\, y \in \R^m \right\} \]
for $\gamma > 0$. Let $R$ given as
$R(v) = \int_{\Omega} g(v(x))\,\dd x$
be a relaxed multibang regularizer. Then for $\gamma > 0$, we define
the function
\[ R_\gamma(v) \coloneqq \int_{\Omega} (e_\gamma g)(v(x))\,\dd x. \]
Moreover, we define the following smoothed
control problems, \eqref{eq:rgamma} for $\gamma > 0$,
\begin{gather}\label{eq:rgamma}
\begin{aligned}
\min_v\ & \mathcal{F}(v) + R_\gamma(v) \\
\text{ s.t.\ }& v \in L^\infty(\Omega,\R^m)
\text{ and }
v(x) \in \conv\{\nu_1,\ldots,\nu_M\} \text{ for a.a.\ } x\in \Omega,
\end{aligned}\tag{\mbox{R$_{\gamma}$}}
\end{gather}
and set (R$_{0}$)\ $\coloneqq$ \eqref{eq:r}.
We summarize the convexity
and differentiability properties and the convergence of minimizers
of the \eqref{eq:rgamma} to minimizers of
\eqref{eq:r} for $\gamma \downarrow 0$ below.
As for \eqref{eq:p} we denote the infimal
value of the \eqref{eq:rgamma} by $\inf \eqref{eq:rgamma}$.

The basic properties of the Moreau envelope yield the following properties.

\begin{proposition}\label{prp:moreau_gr}
\begin{enumerate}
\item\label{itm:moreau_gr_egamma_from_below} Let $u \in \conv V$. Then, $(e_\gamma g)(u) \uparrow g(u)$
for $\gamma \downarrow 0$.
\item\label{itm:moreau_gr_Rgamma_from_below} Let $v \in \mathcal{F}_{\eqref{eq:r}}$. Then, $R_\gamma(v) \uparrow R(v)$ for $\gamma \downarrow 0$.
\item\label{itm:moreau_gr_LC1} Let $\gamma > 0$. Then, the function $e_\gamma g : \conv V \to \R$ is
convex and continuously differentiable with the Lipschitz-continuous derivative.
\item\label{itm:moreau_gr_Rgamma_derivative} Let $\gamma > 0$, and $p \ge 2$. Then, $R_\gamma : L^p(\Omega,\R^m) \to \R$
is differentiable with derivative
$R_\gamma'(v)d  = \int_{\Omega} (e_\gamma g)'(v(x))^T d(x)\,\dd x$
for $d \in L^p(\Omega,\R^m)$.
\end{enumerate}
\end{proposition}
\begin{proof}
Basic properties of the Moreau envelope 
\cite[Sect.\,1.G]{rockafellar2004variational} yield Claim
\ref{itm:moreau_gr_egamma_from_below}.
Claim \ref{itm:moreau_gr_Rgamma_from_below} follows from Beppo--Levi's theorem with Claim \ref{itm:moreau_gr_egamma_from_below}.
Claim \ref{itm:moreau_gr_LC1} follows from \cite[Thm.\,2.26]{rockafellar2004variational}
and the firm nonexpansiveness of the proximal mapping.
The chain rule and differentiability 
properties of Nemitskij operators \cite[Thm.\,7]{goldberg1992nemytskij}
yield Claim \ref{itm:moreau_gr_Rgamma_derivative}.
\end{proof}

\begin{proposition}\label{prp:rgammaGamma}
Let $(\gamma^n)_n \subset [0,\infty)$ satisfy
$\gamma^n \downarrow 0$.
Then, the sequence of functionals $F + R_{\gamma^n}$
on $\mathcal{F}_{\eqref{eq:r}} \subset L^\infty(\Omega,\R^m)$
is $\Gamma$-convergent with limit $F + R$
with respect to weak-$^*$-convergence in $L^\infty$.
\end{proposition}
\begin{proof}
Because $F$ is weak-$^*$-sequentially
continuous, it suffices to prove that $R$ is
a $\Gamma$-limit for the sequence $(R_{\gamma^n})_n$.
Let $L > 0$ denote the Lipschitz constant of
$g$, which exists by virtue of
Proposition \ref{prp:rel_mbreg_props}
\ref{itm:glipschitz}.
Then, we obtain
\begin{gather}\label{eq:diff_RRgn}
R(w) - R_{\gamma^n}(w)
   = \int_{\Omega} g(w(x)) - (e_{\gamma^n}g)(w(x)) \,\dd x
   \le \frac{1}{2} \lambda(\Omega) L^2 \gamma^n
\end{gather}
for all $w \in \mathcal{F}_{\eqref{eq:r}}$,
where $\lambda(\Omega)$ denotes the Lebesgue
measure of $\Omega$. We defer the proof to Section \ref{sec:proof_diff_RRgn}.
Let $v^n$, $v \in \mathcal{F}_{\eqref{eq:r}}$ 
for all $n \in \N$ be
such that $v^n \weakstarto v$
in $L^\infty(\Omega,\R^m)$.
Because $R$ is weak-$^*$-sequentially lower semicontinuous,
it follows that
\[ R(v) \le \liminf R(v^n) \le \liminf
R_{\gamma^n}(v^n) + \frac{1}{2} \lambda(\Omega) L^2 \gamma^n
= \liminf R_{\gamma^n}(v^n),
\]
where the second inequality follows from \eqref{eq:diff_RRgn}
with the choice $w = v^n$.

By virtue of Proposition \ref{prp:moreau_gr}
it follows that
$R_{\gamma^1}(v)
   \le R_{\gamma^n}(v)
   \le R_{\gamma^{n+1}}(v) \le R(v)$
holds
for all $n \in \N$. Combining these estimates, we obtain
$\Gamma$-convergence.
\end{proof}

\begin{corollary}\label{cor:gamma_convergence_epsopt}
Let $(\gamma^n)_n \subset [0,\infty)$ satisfy $\gamma^n \downarrow 0$,
and let $(\varepsilon^n)_n \subset [0,\infty)$ satisfy
$\varepsilon^n \to 0$.
Let $v$, $v^n \in \mathcal{F}_{\eqref{eq:r}}$ for $n \in\N$
satisfy $v^n \weakstarto v$ and
$F(v^n) + R_{\gamma_n}(v^n) < \varepsilon^n + \inf \text{\emph{(R$_{\gamma^n}$)}}$.
Then,  $v$ is a minimizer of \eqref{eq:r}.
\end{corollary}
\begin{proof}
This follows with a standard proof from $\Gamma$-convergence;
see, e.g., \cite{dalMaso1993introduction}.
\end{proof}

\begin{remark}\label{rem:clason_nonsmooth}
This approach differs from the Moreau--Yosida regularization
performed in \cite{clason2018vector}. Therein, the authors
work with a Moreau envelope of the convex conjugate of the relaxed multibang
regularizer $R$, specifically
$(e_\gamma R^*)^*(v) = R(v) + \frac{\gamma}{2}\|v\|^2$.
They improve the convergence of the sequence of minimizing
controls for \eqref{eq:r} to norm-convergence. This is enabled by the strict
convexity that is added by the term $\frac{\gamma}{2}\|v\|^2$. The resulting
regularized multibang regularizer is nondifferentiable, and nonsmooth
techniques are necessary for the optimization;
see also \cite[Rem.\,2.3]{clason2018vector}.
\end{remark}

\section{Algorithmic Framework}\label{sec:algorithm}

We use the properties of relaxed multibang regularizers 
and the optimization problems \eqref{eq:r}
and \eqref{eq:rgamma} to formulate an algorithm
to compute minimizing control sequences for \eqref{eq:p}.
In Section \ref{sec:rounding_algorithms} we provide the necessary
concepts to formulate the second step of
the combinatorial integral approximation, specifically so-called \emph{rounding algorithms}.
In Section \ref{sec:main_algorithm} we integrate these concepts with
the findings from Section \ref{sec:rmbr} into
one algorithm, for which we
prove well-definedness and asymptotics. Practical aspects for
the solution of the involved optimization problems are considered
in Section \ref{sec:practical_considerations}.

\subsection{Rounding Algorithms}\label{sec:rounding_algorithms}

We introduce the concepts of rounding grid
and of order-conserving dissection \cite{manns2020multidimensional}
before defining rounding algorithms.
A rounding grid is a partition of the domain $\Omega$.
An order-conserving domain dissection
is a sequence of
refined rounding grids that satisfies certain regularity properties.

\begin{definition}\label{dfn:order_conserving_domain_dissection}
	Let $\Omega \subset \R^d$ be a bounded domain.
	We call a finite partition $\RG = \left\{T_1,\ldots,T_{N}\right\}
	\in 2^{\Omega}$ of $\Omega$ into $N \in \N$ grid cells
	a \emph{rounding grid}.
	We denote its maximum grid cell volume by
	$\Delta_{\RG} = \max\{ \lambda(T_i) \,|\, i \in [N] \}$.

	We call a sequence of rounding grids
	$(\RG^n)_n \subset 2^{\mathcal{B}(\Omega)}$
	with $\RG^n = \left\{T_1^{n}, \ldots, T_{N^{n}}^{n}\right\}$
	and corresponding maximum grid cell volumes for all $n\in\N$
	an \emph{order-conserving domain dissection of $\Omega$} if
	\begin{enumerate}
		\item\label{itm:max_cellvolume_to_zero} $\Delta_{\RG^n} \to 0$,
		\item\label{itm:recursive_refinement} for all $n$ and all $i \in [N^{n - 1}]$,
		there exist $1 \le j < k \le N^{n}$ such that
		$\bigcup_{l=j}^{k} T_l^{n} = T_i^{n-1}$,
		and
		\item\label{itm:regular_shrinkage} the cells $T^{n}_j$ shrink regularly; that is, there exists
		$C > 0$ such that for each $T_j^{n}$ there exists a Ball
		$B_j^{n}$ such that $T_j^{n}\subset B_j^{n}$
		and $\lambda(T_j^{n}) \ge C \lambda (B_j^{n})$.
	\end{enumerate}
\end{definition}

Definition \ref{dfn:order_conserving_domain_dissection}
\ref{itm:max_cellvolume_to_zero}\ requires that the
maximum volume of the  grid cells of a rounding grid
vanish over the refinements.
Definition \ref{dfn:order_conserving_domain_dissection}
\ref{itm:recursive_refinement}\ requires that
a grid cell of a rounding grid be decomposed into finitely
many grid cells in the next rounding grid and that the
order of the grid cells of a partition be conserved
by the grid cells in which it is decomposed
in all later iterations.
Definition \ref{dfn:order_conserving_domain_dissection}
\ref{itm:regular_shrinkage}\ requires that the grid cells
shrink regularly.
In particular, their eccentricity remains bounded over the iterations.

\begin{example}
Definition \ref{dfn:order_conserving_domain_dissection} is,
for example, satisfied by uniform refinements of a uniform mesh,
where the order of the grid cells is induced by the course of
approximants of space-filling curves through the grid cells
\cite{manns2020multidimensional}.
\end{example}

We introduce the terms \emph{binary} and \emph{relaxed} control 
\cite{manns2020multidimensional} to denote output and
input functions of the rounding algorithms.
Then, we introduce pseudometrics induced by
\emph{rounding grids} to describe the approximation relationship
between input and output of rounding algorithms.

\begin{definition}
	Let $\Omega \subset \mathbb{R}^d$ be a bounded domain. We call
	measurable functions $\omega : \Omega \to \{0,1\}^M$ such that
	$\sum_{i=1}^M \omega_i(x) = 1$ holds a.e.\ \emph{binary controls}.
	We call measurable functions
	$\alpha : \Omega \to [0,1]^M$ such that
	$\sum_{i=1}^M \alpha_i(x) = 1$ holds a.e.\
	\emph{relaxed controls}.
\end{definition}

\begin{definition}
Let $\RG$ be a rounding grid. We define its \emph{induced pseudometric} 
for relaxed controls $\alpha$ and $\beta$ as
\[ d_{\RG}(\alpha,\beta) \coloneqq \max\left\{
   \left\|\int_{\bigcup_{i=1}^k T_i}
    \alpha(x) - \beta(x)\,\dd x\right\|_\infty\,\Bigg|\,
    k \in [N]\right\}.\]
\end{definition}

It is straightforward that $d_{\RG}(\alpha,\beta)$ is a pseudometric
on $L^\infty(\Omega,\R^M)$. Convergence 
with respect to the sequence of $(d_{\RG^n})_n$ induced by an order conserving
domain dissection implies weak-$^*$-convergence in
$L^\infty$ \cite{kirches2020compactness}.
Because $d_{\RG}$ cannot distinguish relaxed controls that have the
same average value per grid cell, we will
approximate an solution of \eqref{eq:rgamma} that is averaged per grid cell.
Thus, for a function $v : \Omega \to \conv V$ and a rounding grid $\RG$,
we define its average per grid cell $\bar{v} : \Omega \to \conv V$ as
\begin{gather}\label{eq:rg_avg}
\bar{v} \coloneqq \sum_{i=1}^N \chi_{T_i} \frac{1}{\lambda(T_i)}\int_{T_i} v(x) \,\dd x.
\end{gather}

We use rounding algorithms as functions that map
relaxed controls to binary controls and provide
an approximation in $d_{\RG}$, which yields the definition below.

\begin{definition}\label{dfn:rounding_algorithm}
A rounding algorithm is a function $\RA$ that 
maps a rounding grid $\RG$
and a relaxed control $\alpha$ to a binary control $\omega$;
that is, $\omega = \RA(\RG,\alpha)$.
In particular, there exists a constant $\theta > 0$
such that $d_{\RG}(\alpha,\omega) \le \theta \Delta_{\RG}$
holds for all relaxed controls $\alpha$, all
rounding grids $\RG$, and $\omega = \RA(\RG,\alpha)$.
\end{definition}

\begin{remark}
Definition \ref{dfn:rounding_algorithm} is satisfied
for sum-up rounding (SUR) \cite{sager2012integer},
next-forced rounding (NFR) \cite{jung2013relaxations}, and the optimization-based 
approaches presented in \cite{jung2013relaxations,bestehorn2020mixed},
which have all been used in the combinatorial integral approximation framework.
\end{remark}

\subsection{Main Algorithm}\label{sec:main_algorithm}

We introduce the algorithmic framework
as Algorithm \ref{alg:sur_based_miocp_approximation}.
Before proving the asymptotics of Algorithm
\ref{alg:sur_based_miocp_approximation} we argue below
that its steps are well defined.
The critical steps that require consideration
are the finite termination of the for loop beginning
in Line \ref{ln:for_begin}
and the computation of $\alpha^n$ in Line \ref{ln:conv_retrieval}.

\begin{algorithm}[ht]
\caption{Approximation of \eqref{eq:p}}\label{alg:sur_based_miocp_approximation}
\textbf{Input:} Order-conserving sequence of rounding grids $(\RG^n)_n$.

\textbf{Input:} Rounding algorithm $\RA$.

\textbf{Input:} Null sequences $\varepsilon^n \downarrow 0$ and $\gamma^n \downarrow 0$.
\begin{algorithmic}[1]
	\State $c^0 \gets 0$
	\For{$n = 0,\ldots$}
		\State\label{ln:relaxation} Compute $v^n$ such that $F(v^n) + R_{\gamma^n}(v^n)
		< \varepsilon^n + \inf \text{(R$_{\gamma^n}$)}$.
		\For{$k=1,\ldots$}\label{ln:for_begin}
			\State Compute avg.\ per grid cell $\bar{v}^n$ (see \eqref{eq:rg_avg}) from $v^n$ on
			rounding grid $\RG^{c^n}$.
			\If{$\|\bar{v}^n - v^n\|_{L^2(\Omega)} < \varepsilon^n$}
				\State \textbf{break} 
			\EndIf
			\State $c^n \gets c^n + 1$
		\EndFor
		\State\label{ln:conv_retrieval} Compute $\alpha^n(x) \coloneqq \argmin\{ \|a\|_2^2 \,|\, a \in G(\bar{v}^n(x))\}$ with $G$ defined in \eqref{eq:Gu}.
		
		\State\label{ln:rounding} Compute binary control $\omega^n = \RA(\RG^{c^n}, \alpha^n)$.
		\State\label{ln:vhat} Compute $\{\nu_1,\ldots,\nu_M\}$-valued
		control
		$\hat{v}^n \coloneqq \sum_{i=1}^M \omega^n_i \nu_i$.
		\State\label{ln:carry_over_grid_counter} $c^{n+1} \gets c^n$.
	\EndFor
\end{algorithmic}
\end{algorithm}

\begin{proposition}\label{prp:welldef}
Let $R$ be a relaxed multibang regularizer. Let the inputs of
Algorithm \ref{alg:sur_based_miocp_approximation} be given.
For all iterations $n \in \N$ of Algorithm 
\ref{alg:sur_based_miocp_approximation} it holds that
\begin{enumerate}
\item\label{itm:forloop} the for loop starting in Line \ref{ln:for_begin}
terminates after finitely many iterations and
\item\label{itm:welldef} $\alpha^n$ is a uniquely defined relaxed control.
\end{enumerate}
\end{proposition}
\begin{proof}
\ref{itm:forloop}.\ Definition
\ref{dfn:order_conserving_domain_dissection}
\ref{itm:regular_shrinkage} implies that
order-conserving domain dissections satisfy
the prerequisites of the
Lebesgue differentiation theorem
\cite[Chap.\,3, Cor.\,1.6\,\&\,1.7]{stein2005real}.
Thus $\bar{v}^n \to v^n$
pointwise a.e.\ for $k \to \infty$.
Because $v^n(x) \in \conv\{\nu_1,\ldots,\nu_M\}$ a.e., which translates to
$\bar{v}^n$ computed in the $k$th iteration,  Lebesgue's
dominated convergence theorem
gives $\bar{v}^n \to v^n$ in $L^2(\Omega)$ for $k \to \infty$.
Since $\varepsilon^n > 0$ for all $n \in \N$,
the termination criterion
$\|\bar{v}^n - v^n\|_{L^2(\Omega)} < \varepsilon^n$
is satisfied after a finite number of iterations of the inner loop.

\ref{itm:welldef}.\ Lemma \ref{lem:gu_unique_welldefined}
gives that $\min \{\|a\|_2^2 \,|\, a \in G(u) \}$ has a unique
minimizer that satisfies $\sum_{i=1}^M a_i = 1$ and $a \ge 0$
for all $u \in \conv V$. Moreover, the function
$\bar{v}^n$ is piecewise constant on the cells of a finite
decomposition of $\Omega$. Third, we observe that $\alpha^n$ is measurable from the fact that $\bar{v}^n$
is measurable and from Lemma \ref{lem:gu_unique_welldefined}.
\end{proof}

Now, we state our main result, the convergence
of the iterates produced by Algorithm \ref{alg:sur_based_miocp_approximation}.
Then, we prove an auxiliary lemma before proving the theorem.

\begin{theorem}\label{thm:main_convergence_result}
Let $R$ be a relaxed multibang regularizer.
Let the inputs of Algorithm \ref{alg:sur_based_miocp_approximation}
be given.
Then, Algorithm \ref{alg:sur_based_miocp_approximation} produces
an infinite sequence of iterates
$v^n$, $\bar{v}^n$, $\hat{v}^n \in L^\infty(\Omega,\R^m)$
such that $v^n$ admits a weak-$^*$-cluster point and
any weak-$^*$-cluster point $v^*$ of $(v^n)_n$
satisfies the following:
\begin{enumerate}
\item\label{itm:minimizing_relaxation} there is a subsequence
      $v^{n_k} \weakstarto v^*$ that minimizes \eqref{eq:r},
\item\label{itm:minimizing_averaged} $\bar{v}^{n_k} \weakstarto v^*$,
\item\label{itm:minimizing_integer} $\hat{v}^{n_k} \weakstarto v^*$,
\item\label{itm:minimizing_F} $F(\hat{v}^{n_k}) \to F(v^*)$,
\item\label{itm:minimizing_R} $R(\hat{v}^{n_k}) \to R(v^*)$, and in particular
\item\label{itm:minimizing_iterats} $(\hat{v}^{n_k})_k$ is a minimzing sequence for \eqref{eq:p}.
\end{enumerate}
\end{theorem}

We prove that $R(\bar{v}^n)$ and $R(\hat{v}^n)$ are close
for $n \to \infty$ as an independent lemma.

\begin{lemma}\label{lem:reg_approx}
Let $R$ be a relaxed multibang regularizer. Let the inputs of
Algorithm \ref{alg:sur_based_miocp_approximation} be given.
There exists $C > 0$ such that for $\bar{v}^n$ and $\hat{v}^n$ computed by
Algorithm \ref{alg:sur_based_miocp_approximation} lines \ref{ln:for_begin}
to \ref{ln:vhat} from $v^n \in \mathcal{F}_{\eqref{eq:r}}$ it holds that $\left|R(\bar{v}^n) - R(\hat{v}^n)\right| \le C \Delta_{\RG^{c^n}}$.
The constant $C > 0$ does not depend on the rounding grids $\RG^{c^n}$.
\end{lemma}
\begin{proof}
Because we consider a fixed iteration $n$, we omit the index $n$ for the
functions as well as the index $c^n$ for the rounding grid in this proof.
From the construction of $\alpha$ in Line \ref{ln:conv_retrieval} we obtain
$R(\bar{v})
   = \int_{\Omega} g(\bar{v}(x))\,\dd x
   = \int_{\Omega} \sum_{i=1}^M \alpha_{i}(x) g_i\,\dd x$.
Because $\omega = \RA(\RG,\alpha)$, we have that $\omega$ is a binary control,
which implies
\[ R(\hat{v})
   = \int_{\Omega} g(\hat{v}(x))\,\dd x
   = \int_{\Omega} \sum_{i=1}^M \omega_i(x) g(\nu_i)\,\dd x
   = \int_{\Omega} \sum_{i=1}^M \omega_i(x) g_i\,\dd x, \]
where the last equality follows from Lemma \ref{lem:gu_unique_welldefined}.
We conclude
\[ \left|R(\bar{v}) - R(\hat{v})\right|
   \le \sum_{i=1}^M |g_i| \left|\int_{\Omega}\alpha_i(x) - \omega_i(x)\,\dd x\right|
   \le \sum_{i=1}^M |g_i| d_{\RG}(\alpha,\omega), \]
where the first inequality follows from the triangle inequality and the
second inequality by definition of $d_{\RG}$.
Because of Definition \ref{dfn:rounding_algorithm},
the claim follows with the choice $C \coloneqq \sum_{i=1}^M |g_i| \theta$.
\end{proof}

We are ready to prove our main result.

\begin{proof}[Proof of Theorem \ref{thm:main_convergence_result}]
For the statements of Theorem \ref{thm:main_convergence_result}
to be well defined, we require that the
inner loop (indexed by $k$) of Algorithm \ref{alg:sur_based_miocp_approximation}
terminates finitely for all $n \in\N$.
This follows from Proposition \ref{prp:welldef}.
The existence of weak-$^*$-cluster points
follows from the boundedness of the sequence.
We prove the claims one by one. 

\ref{itm:minimizing_relaxation}.\ The minimization property
of weakly-$^*$-convergent subsequences follows directly from Corollary 
\ref{cor:gamma_convergence_epsopt}.

\ref{itm:minimizing_averaged}.\ follows from the finite termination of
the inner loop, in particular $\|\bar{v}^{n_k} - v^{n_k}\|_{L^2(\Omega)} \to 0$,
and $v^{n_k} \weakstarto v^*$.

\ref{itm:minimizing_integer}.\ We consider the construction of $\hat{v}^{n_k}$
and obtain
\begin{align*}
	\hat{v}^{n_k}
    = \sum_{i=1}^M \omega^{n_k}_i \nu_i
    = \sum_{i=1}^M (\omega^{n_k}_i - \alpha^{n_k}_i) \nu_i
      + \sum_{i=1}^M \alpha^{n_k}_i \nu_i
    = \sum_{i=1}^M (\omega^{n_k}_i - \alpha^{n_k}_i) \nu_i + \bar{v}^{n_k}.
\end{align*}
Then, we test with $f \in L^1(\Omega,\R^m)$ and apply
\cite[Lem.\,4.4]{manns2020multidimensional} to obtain that the
difference term vanishes weakly-$^*$ for $k\to \infty$.
We deduce
\[ \lim_{k\to\infty} \langle v^{n_k},f \rangle_{L^1,L^\infty}
   = \lim_{k\to\infty} \langle \bar{v}^{n_k},f\rangle_{L^1,L^\infty}
   = \langle v^*, f\rangle_{L^1,L^\infty}. \]

\ref{itm:minimizing_F}.\ follows from \ref{itm:minimizing_integer}.\
and the weak-$^*$-sequential continuity of $F$.

\ref{itm:minimizing_R}.\ We estimate
\[ \left|R(\hat{v}^{n_k}) - R(v^*)\right|
   \le \left|R(\hat{v}^{n_k}) - R(\bar{v}^{n_k})\right|
     + \left|R(\bar{v}^{n_k}) - R(v^{n_k})\right|
     + \left|R(v^{n_k}) - R(v^*)\right| \]
by means of the triangle inequality.
Because the rounding algorithm is executed on grids of an order-conserving
domain dissection and $c^{n_k} \to \infty$, the first term tends to zero
by virtue of Lemma \ref{lem:reg_approx}.
The finite termination of the inner loop and $\varepsilon^{n_k}\to 0$
imply that the second term tends to zero.
It remains to show $\left|R(v^{n_k}) - R(v^*)\right| \to 0$.
To this end, we estimate
\begin{align*}
\left|R(v^{n_k}) - R(v^*)\right|
&\le \left|R_{\gamma_{n_k}}(v^{n_k}) - R(v^{n_k})\right| + \left|R_{\gamma_{n_k}}(v^{n_k}) - R(v^*)\right|
\end{align*}
by means of the triangle inequality.
The first term tends to zero by virtue of \eqref{eq:diff_RRgn}.
The second term tends to zero by combining \ref{itm:minimizing_relaxation}.\
with \ref{itm:minimizing_F}.
\end{proof}

\begin{remark}
Minimizers of \eqref{eq:r} are generally
fractional-valued. Thus, an approximation with discrete-valued controls
can only achieve weak-$^*$-convergence in $L^p$-spaces
and norm-convergence is only conceivable for discrete-valued minimizers.
\end{remark}

\subsection{Practical Considerations}\label{sec:practical_considerations}

The fact that we require $v^{n}$ to be $\varepsilon^n$-optimal for
\eqref{eq:r} allows us to replace the infinite-dimensional optimization
problem (R$_{\gamma^n}$) by successively refined discretized
finite-dimensional problems
using, for example,  the argument in \cite{haslinger2015topology}.

In Algorithm \ref{alg:sur_based_miocp_approximation}
ln.\ \ref{ln:conv_retrieval}, the bilevel optimization problem
\begin{gather}\label{eq:bilevel}
\begin{aligned}
\min\ \|a\|_2^2\ \text{ s.t.\ }
a \in \argmin\left\{\sum_{i=1}^M a_i g_i\,\Big|\, \sum_{i=1}^M a_i\nu_i = u,
\sum_{i=1}^M a_i = 1, a \ge 0\right\}.
\end{aligned}
\end{gather}
has to be solved for different inputs $u \in \conv V$.
Proposition \ref{prp:rel_mbreg_props} gives that the lower-level
problem is an \ac{LP} that has a nonempty bounded feasible set.
Consequently, strong duality for \acp{LP} holds, and we
may rewrite \eqref{eq:bilevel} equivalently as
\begin{gather}\label{eq:bilevel_as_qp}
\begin{aligned}
\min_{a,y,z}\ a^Ta
\text{ s.t.\ } &
\left\{
\begin{aligned}
a^T c - y^T u - z  &= 0,\ 
A a = u,\ 
\mathbbm{1}^T_M a = 1,\\
y^T A + z\mathbbm{1}^T_M &\le c^T \\
a_i &\ge 0  \text{ for } i \in \{1,\ldots,M\}, \\
y_i &\in \R \text{ for } i \in \{1,\ldots,m\}, \\
z   &\in \R.
\end{aligned}
\right.
\end{aligned}
\end{gather}
In \eqref{eq:bilevel_as_qp},
$A \in \R^{m\times M}$ is the matrix that consists of the
vectors $\nu_i$ for $i \in [M]$ as columns,
$c \in \R^M$ is the vector that consists of the scalars
$g_i$ for $i \in [M]$ as components,
and $\mathbbm{1}_M$ is the vector in $\R^M$ that
equals $1$ in all entries.
It is a convex quadratic program
with $M + m + 1$ variables, which can be solved with
quadratic programming techniques.

The function $\bar{v}^n$ is constant per grid cell on the finitely
many grid cells that constitute $\RG^{c^n}$.
Thus, \eqref{eq:bilevel_as_qp} has only to be solved
finitely many times per iteration. In fact, the proof of Theorem \ref{thm:main_convergence_result}
can be carried out without the intermediate step of computing
$\bar{v}^n$. However, defining the function $\alpha^n$
as the solution of \eqref{eq:Gu} or \eqref{eq:bilevel_as_qp}
with $u = v^n(x)$ for a.a.\ $x \in \Omega$ cannot be implemented
directly.

It is possible to replace the
smoothed regularization and the smooth optimization in Algorithm 
\ref{alg:sur_based_miocp_approximation} ln.\ \ref{ln:relaxation} by the
nonsmooth problems and the semismooth Newton method presented in
\cite{clason2018vector}.
Our analysis is not restricted to Moreau envelopes
to approximate relaxed multibang regularizers smoothly.
We provide a feasible alternative for the scalar-valued case
$m = 1$ in Section \ref{sec:1dim_subs} in closed form.

\section{Computational Examples}\label{sec:computational_examples}

We provide two examples to demonstrate the efficacy of the algorithmic
framework and validate our arguments computationally.
The first example is a signal reconstruction problem with a
one-dimensional control input \cite{buchheim2012effective,kirches2020compactness}.
The second example is the Lotka--Volterra problem from
the benchmark library
\href{https://mintoc.de/index.php/Lotka_Volterra_fishing_problem}{\textsc{Mintoc}}
\cite{sager2012benchmark}.
It has been used frequently to evaluate algorithms for the
combinatorial integral approximation
\cite{sager2005numerical,sager2006numerical,bestehorn2019switching,bestehorn2020mixed}.
We modify the problem such that it has a two-dimensional control input.

For both examples, we have used rounding algorithms
that have the additional property
that for all grid cells $T \in \RG^{c^n}$ it holds that
$\int_{T} \alpha^n(x)\,\dd x = 0$ implies $\int_{T} \omega^n(x)\,\dd x = 0$.
This can prevent  $\omega^n$ from spontaneously switching
on a control value that is
\emph{far away} from the values of the continuous relaxation at this spot.

\subsection{Signal Reconstruction Problem}

The aforementioned
signal reconstruction problem from \cite[Sect.\,5]{kirches2020compactness}
with relaxed multibang regularizer is
\begin{gather}\label{eq:spp}
\begin{aligned}
\inf_{y,v}\ &\frac{1}{2} \int_{t_0}^{t_f} (y(t) - f(t))^2\,\dd t
   + \eta \int_{t_0}^{t_f} g(v(t))\,\dd t \eqqcolon  J(v)\\
\text{s.t.}\ & y = k * v\text{ and }
v(t) \in \{\nu_1,\ldots,\nu_M\} \text{ for a.a.\ } t \in (t_0,t_f),
\end{aligned}\tag{SRP}
\end{gather}
where $t_0 = -1$, $t_f = 1$,
$\nu_1 = -1$, $\nu_2 = -0.25$, $\nu_3 = 0$, $\nu_4 = 0.35$, and
$\nu_5 = 1$ and $g$ is defined as in Definition \ref{dfn:rel_mbreg} with
$g_1 = 1$, $g_2 = 0.125$, $g_3 = 0$, $g_4 = 0.175$, and $g_5 = 1$.
We choose $\eta = 0.01$. The continuous relaxation of
\eqref{eq:spp} is convex.
We discretize the convolution with
a Gau\ss--Legendre quadrature and use a
piecewise constant control ansatz for the continuous relaxation.
We smooth the function $g$ as proposed in Section \ref{sec:1dim_subs}. \textsc{Scipy's} \cite{2020SciPy-NMeth} implementation of \textsc{L-BFGS-B}
\cite{branch1999subspace,voglis2004rectangular}
is used to solve the smoothed continuous relaxation.
We use both the optimization-based approach \textsc{SCARP}
\cite{bestehorn2019switching,bestehorn2020mixed}
and \textsc{SUR} in the version of \cite{kirches2020approximation}
as rounding algorithms in Algorithm \ref{alg:sur_based_miocp_approximation}.
We choose $\theta = 2$,
for which \textsc{SCARP} satisfies
the prerequisites of Definition \ref{dfn:rounding_algorithm}
with the same constant as \textsc{SUR}.
This follows from the bounds in \cite{manns2020approximation}.

We show how $\inf \eqref{eq:spp}$ gets approximated from
below by the minimizers of the smoothed continuous relaxation and with the iterates produced in
Algorithm \ref{alg:sur_based_miocp_approximation} ln.\ \ref{ln:vhat}. To this end, we use the same fine
discretization and high accuracy for
Algorithm \ref{alg:sur_based_miocp_approximation}
ln.\ \ref{ln:relaxation}
for all iterations. We have run Algorithm
\ref{alg:sur_based_miocp_approximation} for nine iterations.
We provide the values of $\Delta^n$, $N^n$, $\varepsilon^n$, $\gamma^n$ as well as
the relative difference in the objective between the overestimator $J(\hat{v}^n)$
and the underestimator $J_{\gamma^n}(v^n)$, the infimum for the current discretization and smoothing, in Table \ref{tbl:signal}.
The relative objective error tends to zero over the iterations, which indicates
that the computed iterates $v^n$ of the discretized continuous relaxations tend
to a minimizer and the discrete-valued iterates $\hat{v}^n$ converge 
weakly-$^*$ to the same minimizer.

The smoothed relaxed and discrete controls for iterations $1$, $5$, and $9$ are depicted
in Figure \ref{fig:signalctrls}.
The first row shows how discrete-valued controls
are promoted by the relaxed multibang regularizer.
Comparing the center row with the bottom row, one can observe
the reduced switching behavior that is due to the use of \textsc{SCARP}
instead of \textsc{SUR}.
The objective values for the smoothed relaxed controls and discrete controls
are displayed over the iterations in Figure \ref{fig:signalobj}. As predicted by our analysis,
the gap between the lower bound given by the
minimal objective value of $\eqref{eq:rgamma}$ and the upper bound given by the objective value of the discrete control tends to zero.

\begin{table}[t]
\begin{center}
\caption{Output of nine iterations of Algorithm \ref{alg:sur_based_miocp_approximation} 
applied to \eqref{eq:spp}.}\label{tbl:signal}
\begin{tabular}{rrccccc}
\hline
\multicolumn{1}{r}{\multirow{2}{*}{It.}}
& \multicolumn{1}{r}{\multirow{2}{*}{$N^n$}}
& \multicolumn{1}{c}{\multirow{2}{*}{$\Delta^n$}}
& \multicolumn{1}{c}{\multirow{2}{*}{$\varepsilon^n$}}
& \multicolumn{1}{c}{\multirow{2}{*}{$\gamma^n$}}
& \multicolumn{2}{c}{$\frac{J(\hat{v}^n) - J_{\gamma^n}(v^n)}{J_{\gamma^n}(v^n)}$} \\
\multicolumn{1}{r}{} 
& \multicolumn{1}{r}{}
& \multicolumn{1}{c}{}
& \multicolumn{1}{c}{}
& \multicolumn{1}{c}{}
& \multicolumn{1}{c}{\textsc{SUR}}
& \multicolumn{1}{c}{\textsc{SCARP}} \\
\hline
1  &    16  &  1.2500e-01  &  1.000e+00  &  4.000e-01  &  2.3736e+00  &  9.5216e+00 \\
2  &    32  &  6.2500e-02  &  5.000e-01  &  2.000e-01  &  7.5314e-01  &  1.5831e+00 \\
3  &    64  &  3.1250e-02  &  2.500e-01  &  1.000e-01  &  4.7885e-01  &  6.8392e-01 \\
4  &   128  &  1.5625e-02  &  1.250e-01  &  5.000e-02  &  1.9752e-01  &  1.4560e-01 \\
5  &   256  &  7.8125e-03  &  6.250e-02  &  2.500e-02  &  8.4494e-02  &  1.2059e-01 \\
6  &   512  &  3.9062e-03  &  3.125e-02  &  1.250e-02  &  3.7259e-02  &  4.6550e-02 \\
7  &  1024  &  1.9531e-03  &  1.563e-02  &  6.250e-03  &  1.4573e-02  &  1.2344e-02 \\
8  &  2048  &  9.7656e-04  &  7.813e-03  &  3.125e-03  &  7.6772e-03  &  6.4613e-03 \\
9  &  4096  &  4.8828e-04  &  3.906e-03  &  1.563e-03  &  3.9408e-03  &  3.1779e-03 \\
\hline
\end{tabular}
\end{center}
\end{table}

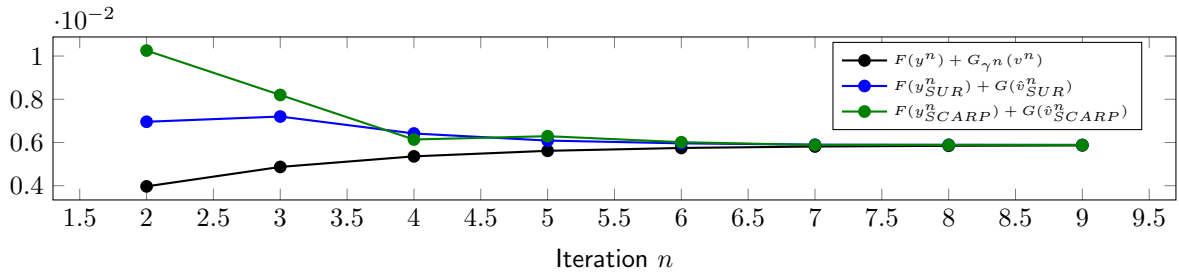
\begin{figure}[h]
	\pgfplotstableread[col sep=semicolon]{data/alg2_obj_data_2.50e-02.csv}\dataobj
	\pgfplotsset{width=\linewidth,height=3.75cm}
	\begin{tikzpicture}
	\begin{axis}[axis background/.style={fill=gray!0},
	xlabel={Iteration $n$},
	legend columns=1,
	width=\textwidth,
	height=3.75cm,
	legend cell align={left},
	legend pos=north east,
	legend style={font=\tiny,/tikz/every even column/.append style={column sep=0.2cm}}]
	\addplot[thick,color=black,solid,mark=*,mark options={solid}]  table [skip first n=0, x index=0, y index=1]{\dataobj};
	\addlegendentry{$F(y^n) + G_{\gamma^n}(v^n)$};
	\addplot[thick,color=blue,solid,mark=*,mark options={solid}]  table [skip first n=0, x index=0, y index=3]{\dataobj};
	\addlegendentry{$F(y^n_{SUR}) + G(\hat{v}^n_{SUR})$};
	\addplot[thick,color=darkgreen,solid,mark=*,mark options={solid}]  table [skip first n=0, x index=0, y index=4]{\dataobj};
	\addlegendentry{$F(y^n_{SCARP}) + G(\hat{v}^n_{SCARP})$};
	\end{axis}
	\end{tikzpicture}
	\caption{Optimal objective of the smoothed relaxation
	and objective values of \eqref{eq:spp}
	for the discrete-valued controls $\hat{v}^n_{SUR}$ and $\hat{v}^n_{SCARP}$
	over the iterations $n = 2,\ldots,9$.}\label{fig:signalobj}
\end{figure}

\begin{figure}[t]
\begin{minipage}{.33\textwidth}
\includegraphics[width=\linewidth,height=3cm]{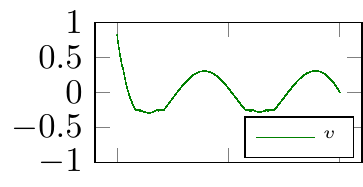}\\
\includegraphics[width=\linewidth,height=3cm]{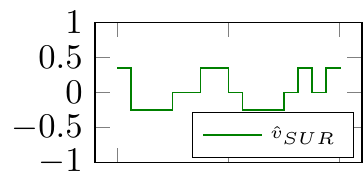}\\
\includegraphics[width=\linewidth,height=4cm]{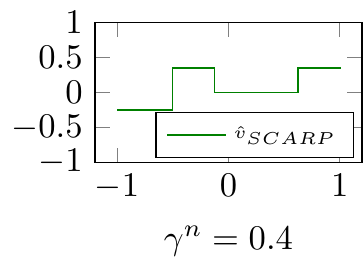}
\end{minipage}\hfill
\begin{minipage}{.33\textwidth}
\includegraphics[width=\linewidth,height=3cm]{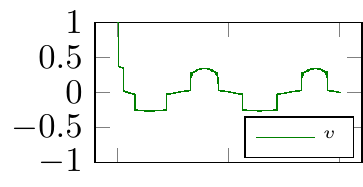}\\
\includegraphics[width=\linewidth,height=3cm]{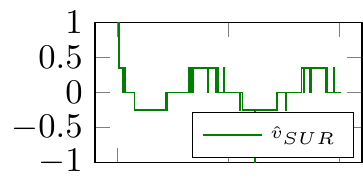}\\
\includegraphics[width=\linewidth,height=4cm]{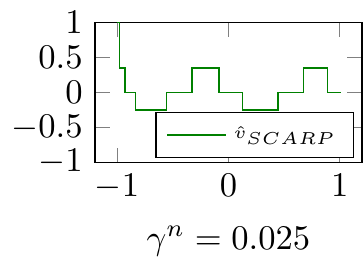}
\end{minipage}\hfill
\begin{minipage}{.33\textwidth}
\includegraphics[width=\linewidth,height=3cm]{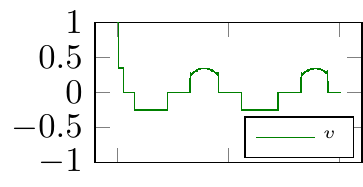}\\
\includegraphics[width=\linewidth,height=3cm]{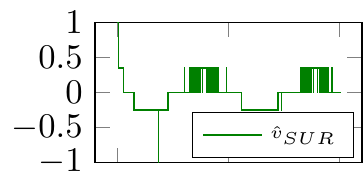}\\
\includegraphics[width=\linewidth,height=4cm]{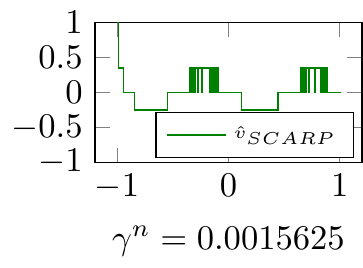}
\end{minipage}
\caption{Computed controls in iterations $n = 1$,$5$,$9$ (left to right)
in ln.\ \ref{ln:relaxation} (top) and in
ln.\ \ref{ln:vhat} for SUR (center) and SCARP (bottom).}\label{fig:signalctrls}
\end{figure}

\subsection{Lotka--Volterra Problem}

The Lotka--Volterra problem \cite{sager2012benchmark}
with relaxed multibang regularizer 
for a two-dimensional discrete control input
is
\begin{gather}\label{eq:lvp}
\begin{aligned}
\inf_{y,v}\ & \int_{t_0}^{t_f} \left\|y(t) - \begin{pmatrix} 1 & 1 \end{pmatrix}^T\right\|_2^2
+ \eta \int_{t_0}^{t_f} g(v(t))\,\dd t \eqqcolon J(v)\\
\text{s.t.}\ &
\left\{
\begin{aligned}
	\dot{y}_1(t) &= y_1(t) - y_1(t)y_2(t) - y_1(t)v_1(t) \text{ for a.a.\ } t \in (t_0,t_f),\\
	\dot{y}_2(t) &= -y_2(t) + y_1(t)y_2(t) - y_2(t)v_2(t) \text{ for a.a.\ } t \in (t_0,t_f),\\
	y(t_0) &= \begin{pmatrix} 0.5 & 0.7 \end{pmatrix}^T,\\
	v(t) &\in \left\{ 
		\nu_1,\nu_2,\nu_3,\nu_4,\nu_5
	\right\} \text{ for a.a.\ } t \in (t_0,t_f),
\end{aligned}
\right.
\end{aligned}\tag{LVP}
\end{gather}
where $t_0 = 0$, $t_f = 12$,
$\nu_1 = \begin{pmatrix} 0 & -0.1 \end{pmatrix}^T$,
$\nu_2 = \begin{pmatrix} 0.05 & 0.0 \end{pmatrix}^T$,
$\nu_3 = \begin{pmatrix} 0.4 & -0.1 \end{pmatrix}^T$,
$\nu_4 = \begin{pmatrix} 0.0 & 0.1 \end{pmatrix}^T$,
and $\nu_5 = \begin{pmatrix} 0.4 & 0.1 \end{pmatrix}^T$
and $g$ is defined as in Definition \ref{dfn:rel_mbreg} with
$g_1 = 2$, $g_2 = 0$, $g_3 = 1$, $g_4 = 2$, and $g_5 = 0.1$.
We choose $\eta = 0.005$. We discretize the initial value problem and the objective and solve the continuous relaxation with
the software package \textsc{CasADi}
\cite{andersson2019casadi} with \textsc{Ipopt}
\cite{wachter2006implementation}
as the nonlinear programming solver.

Deriving a closed-form expression
for the Moreau envelopes of $g$ is difficult.
Therefore, we have computed the values of $g$ and $g_{\gamma}$
on a fine grid that discretizes $\conv V = [0.0,0.4] \times [-0.1,0.1]$
and interpolated them.
Figure \ref{fig:2drmbr} shows the function $g$ and its Moreau envelopes $g_\gamma$ for the
choices $\gamma = 5\cdot 10^{-3}$ and $\gamma = 10^{-3}$.
 
\begin{figure}
\begin{center}
\includegraphics[width=\linewidth]{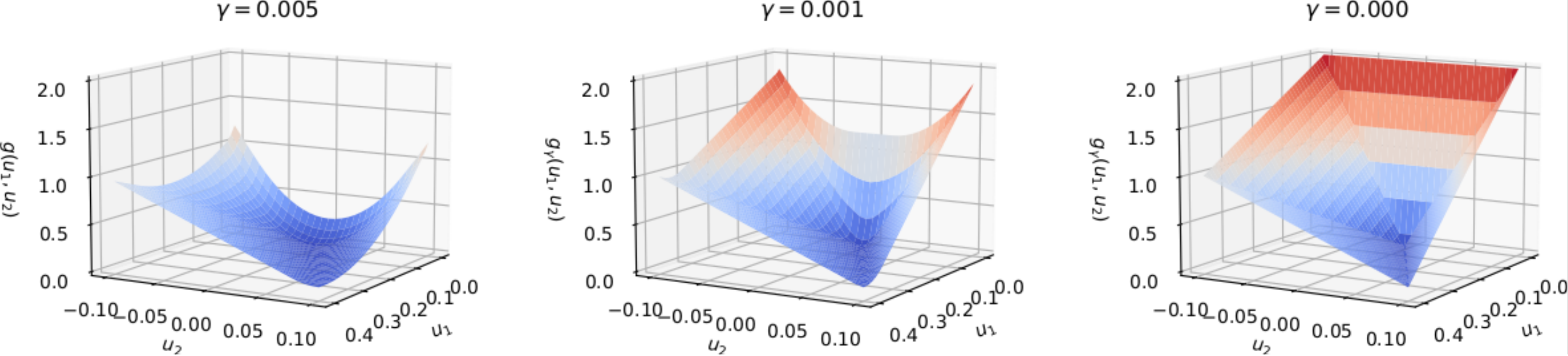}
\end{center}
\caption{Two-dimensional multibang relaxed regularizer (right)
and its Moreau envelopes for $\gamma = 5\cdot 10^{-3}$ (left)
and $\gamma = 10^{-3}$ (center).}\label{fig:2drmbr}
\end{figure}

The Lotka--Volterra problem may be nonconvex, and we cannot expect
more than convergence of the solution of the continuous relaxation to a local
minimizer. Therefore, we disregard the global optimality condition implied
by Algorithm  \ref{alg:sur_based_miocp_approximation}
ln.\ \ref{ln:relaxation} for this problem.
Again, we use \textsc{SCARP}
\cite{bestehorn2019switching,bestehorn2020mixed}
as rounding algorithm in Algorithm \ref{alg:sur_based_miocp_approximation}
and choose $\theta = 2$ with respect to Definition \ref{dfn:rounding_algorithm}.
We mimic driving $\varepsilon^n$ to zero by refining
the discretization in every iteration. We optimize
over piecewise constant ansatz functions for the controls and use
the control discretization grid as rounding grid.
Therefore, we always have $\|v^n - \bar{v}^n\|_{L^2} = 0$.

We have run the algorithm for six iterations.
We provide the values of $\Delta^n$, $N^n$, $\gamma^n$ 
and the relative
difference between the overestimator $J(\hat{v}^n)$
and the underestimator $J_{\gamma^n}(v^n)$
for the current discretization in Table \ref{tbl:lotka}.
The relative difference
tends to zero over the iterations, which indicates
that the computed iterates $v^n$ of the discretized continuous relaxations tend
to a local minimizer and corresponding
weak-$^*$-convergence of the discrete-valued iterates
$\hat{v}^n$.
The $L^2$-difference between relaxed and discrete control
$\|v^n - \hat{v}^n\|_{L^2}$ decreases with less smoothing over the iterations, indicating that the relaxed multibang regularizer
promotes discrete-valued controls.
The relaxed and discrete controls in iterations $2$, $4$ and $6$ are displayed in 
Figure \ref{fig:2dalgo}.

\begin{table}[h]
\begin{center}
\caption{Output of six iterations of Algorithm \ref{alg:sur_based_miocp_approximation} 
applied to \eqref{eq:lvp}.}\label{tbl:lotka}
\begin{tabular}{rrcccc}
\hline
Iteration & $N^n$ & $\Delta^n$ & $\gamma^n$ & $\frac{J(\hat{v}^n) - J_{\gamma^n}(v^n)}{J_{\gamma^n}(v^n)}$ &
$\|v^n - \hat{v}^n\|_{L^2}$ \\
\hline
1 &  16 & 7.5000e-01 & 3.1250e-01 & 2.9629e+00 & 4.0609e-01 \\
2 &  32 & 3.7500e-01 & 6.2500e-02 & 9.5257e-01 & 4.7660e-01 \\
3 &  64 & 1.8750e-01 & 1.2500e-02 & 5.0805e-01 & 3.9329e-01 \\
4 & 128 & 9.3750e-02 & 2.5000e-03 & 3.9579e-02 & 2.4475e-01 \\
5 & 256 & 4.6875e-02 & 5.0000e-04 & 2.4117e-02 & 2.1208e-01 \\
6 & 512 & 2.3438e-02 & 1.0000e-04 & 6.9964e-03 & 1.8743e-01 \\
\hline
\end{tabular}
\end{center}
\end{table}

\begin{figure}[h]
\begin{center}
\includegraphics[width=\linewidth]{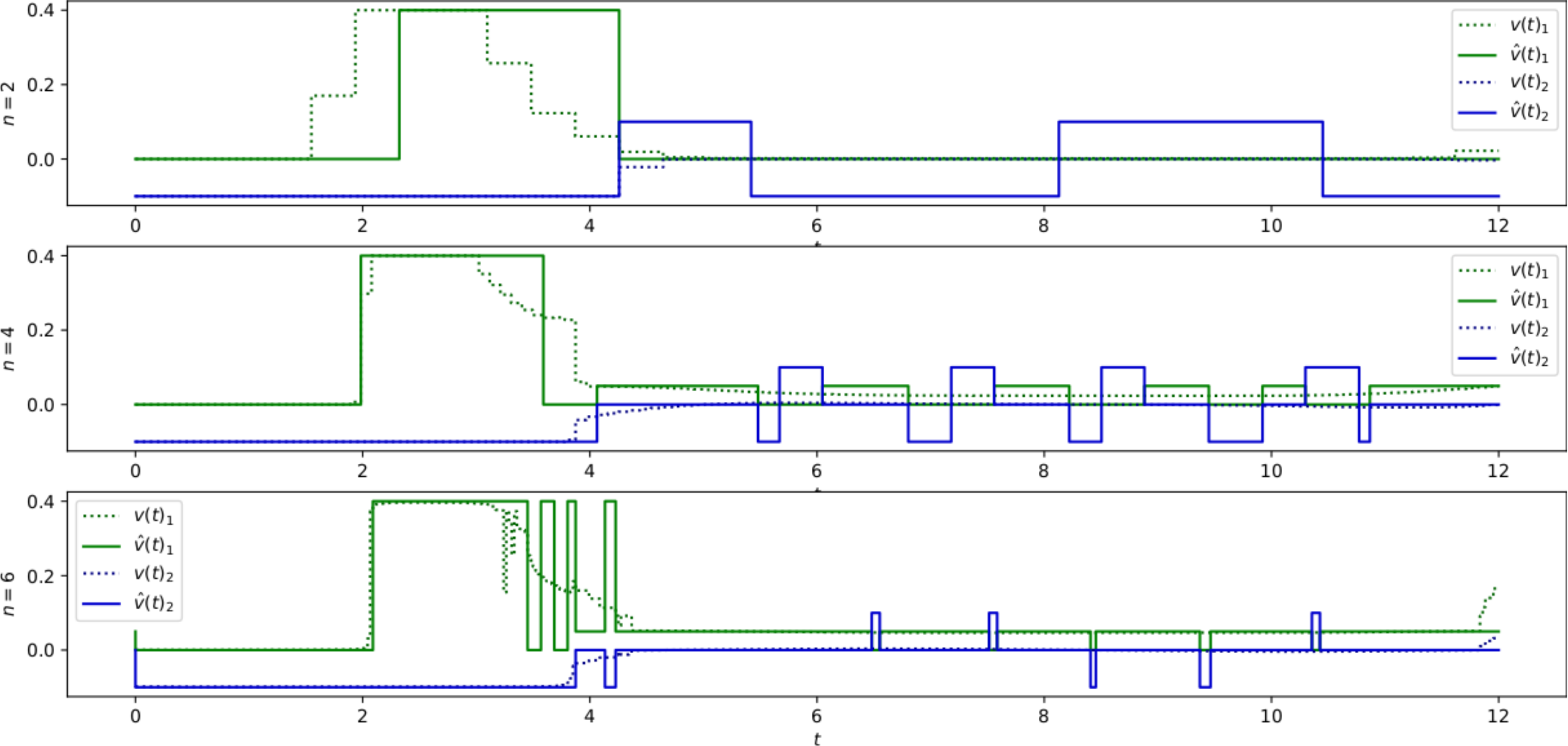}
\end{center}
\caption{Components ($1$ green, $2$ blue) of the controls $v^n$ (dotted) and $\hat{v}^n$
(solid) produced by Algorithm \ref{alg:sur_based_miocp_approximation} in iterations $n = 2,4,6$.}\label{fig:2dalgo}
\end{figure}

\section{Conclusion}

Relaxed multibang regularizers are suitable for regularizing
the integer optimal control problems that can be treated with the
combinatorial integral approximation decomposition.
They can be integrated in the algorithmic framework,
and their nonsmoothness can be alleviated by using
Moreau envelopes.
Two test problems demonstrate the efficacy of the extended
algorithmic framework.

The presented approach can reduce
switching costs in practice by promoting $V$-valued controls
in \eqref{eq:r}, which is indicated by the numerical results; also
compare Figure \ref{fig:signalctrls} with Figure 4 in 
\cite{kirches2020compactness} and consider the results in 
\cite{clason2014multi,clason2016convex}.
Because of the weak-$^*$ approximation, high-frequency switching cannot
be avoided for fine rounding grids if the (optimal) relaxed control
is fractional-valued.
Total variation penalties are an intuitive choice
to avoid high-frequency switching in optimal control;
see also \cite{clason2018total}.
They cannot be modeled as relaxed multibang regularizers, however;
see also \cite[Rem.\,4]{bestehorn2020mixed}.
Switching costs and constraints on dwell times
can be included in the second step of the combinatorial integral 
approximation; see \cite{bestehorn2020mixed,zeile2020mixed}. If
switching costs need to be limited a priori, a general theory
is not at hand and, in the context of
the combinatorial integral approximation,
a gap between $\inf \eqref{eq:p}$ and $\min \eqref{eq:r}$ has to
be accepted so far. Finally, regularization terms may not always make
sense.
If the infimal value of the term $F$ needs to be approximated closely
and high-frequency switching controls can be generated, then
a choice $R \neq 0$ may counteract this goal and increase
the lowest achievable value of $F$.

\section*{Acknowledgments}

This work was supported by the Applied Mathematics activity within the U.S. Department of Energy, Office of Science, Advanced Scientific Computing Research Program (ASCR), under contract number DE-AC02-06CH11357.
We are greatful to Bart van Bloemen Waanders, Drew Kouri,
Denis Ridzal, and Cynthia Phillips
for the fruitful discussions on the topic of this paper
during and in the follow-up of
a visit of the author to Sandia National Laboratories.

\begin{acronym}[Bash]
 \acro{LP}{linear program}
\end{acronym}

\appendix
\section{Auxiliary Proofs}

\subsection{Proof of Proposition \ref{prp:counter_statement_strict_convex}}
\label{sec:proof_prp_counter_strict_cvx}

\begin{proof}
Since the Carath\'{e}odory conditions are satisfied for $g$,
the Nemytskii operator induced by $g$ is a bounded and continuous map
from $L^2(\Omega,\R^m)$ to $L^2(\Omega)$ and thus
$R \in C(L^2(\Omega,\R^m),[0,\infty))$;
see \cite[Sect.\,10.3.4]{renardy2006introduction}.

We split $R(\bar{v})$ as follows:
\[ R(w) = \int_A g(\bar{v}(x))\,\dd x
   + \int_{A^c} g(\bar{v}(x))\,\dd x.
\]
Let $x \in \Omega$. Let $i \in [M]$ and $n \in \N$.
Then, we can define
\[ \beta^n_i(x) \coloneqq \left\{\begin{aligned}
	1 & \text{ if } v^n(x) = \nu_i, \\
	0 & \text{ else}.
\end{aligned}
\right.
\]
Since the $v^n$ are measurable, so are the $\beta^n$.
The sequence $(\beta^n)_n$ is bounded in the space $L^\infty(\Omega, \R^M)$
and thus admits a weakly-$^*$-convergent subsequence
with limit $\alpha \in L^\infty(\Omega, \R^M)$.
Moreover, since $v^n = \sum_{i=1}^M \beta^n_i \nu_i$,
we obtain
\[ v^n = \sum_{i=1}^M \beta^n_i \nu_i
   \weakstarto \sum_{i=1}^M \alpha^n_i \nu_i = \bar{v}, 
\]
where the last identity follows from the uniqueness of the limit.
From the fact that the $\beta^n$ are $\{0,1\}^M$-valued, we deduce
that $0 \le \alpha_i(x)$ and
$\sum_{i=1}^M \alpha^n_i(x) = 1$ for a.a.\ $x \in \Omega$.
In words, $\alpha$ constitutes a vector of convex coefficients a.e.
Moreover, the fact that the $\beta^n$ are $\{0,1\}^M$-valued
also implies $g(v^n(x)) = \sum_{i=1}^M \beta^n_i(x) g(\nu_i)$
for a.a.\ $x \in \Omega$.

Thus, we may conclude
\begin{gather}\label{eq:strict_convexity_estimate_part_1}
R(v^n)
= \sum_{i=1}^M \int_{\Omega} \beta^n_i(x) g(\nu_i)\,\dd x
\to \sum_{i=1}^M \int_{\Omega} \alpha_i(x) g(\nu_i)\,\dd x.
\end{gather}

Let $D \subset \Omega$ be measurable. Then,
\begin{gather}\label{eq:integrated_convexity_inequality}
\sum_{i=1}^M \int_{D} \alpha_i(x) g(\nu_i)\,\dd x \ge \int_{D} g(\bar{v}(x))\,\dd x
\end{gather}
because $\alpha$ constitutes a vector of convex coefficients a.e.,
and $g$ is convex.

Now, we note that $\alpha(x) \notin \{0,1\}^M$
for a.a.\ $x \in A$ because $\bar{v}(x) \in \conv V \setminus V$
for a.a.\ $x \in A$. We show that there exists $\varepsilon > 0$
such that
\[ \sum_{i=1}^M \alpha_i(x) g(\nu_i)
> g\left(\sum_{i=1}^M \alpha_i(x) \nu_i\right)
+ \varepsilon = g(\bar{v}(x)) + \varepsilon \]
for all $x \in B$ for a set $B \subset A$ such that $\lambda(B) > 0$.
To see this, we assume the converse and obtain that
\[ \sum_{i=1}^M \alpha_i(x) g(\nu_i)
= g\left(\sum_{i=1}^M \alpha_i(x) \nu_i\right)
= g(\bar{v}(x)) \]
holds for a.a.\ $x \in A$. But since $g$ is strictly convex, this means
that $\alpha(x) \in \{0,1\}^M$ for a.a. $x\in A$, which is a contradiction.

We conclude that
\[ \int_{B} \sum_{i=1}^M \alpha_i(x) g(\nu_i)
   > \int_B g(w(x))\,\dd x + \varepsilon \lambda(B).\]
Combining this estimate with \eqref{eq:integrated_convexity_inequality}
for the choice $D = B^c$ and inserting both estimates into
\eqref{eq:strict_convexity_estimate_part_1}, we get 
\begin{align*}
\liminf R(v^n) &= \sum_{i=1}^M \int_{B^c}
                  \alpha_i(x) g(\nu_i)\,\dd x
      + \sum_{i=1}^M \int_{B} \alpha_i(x) g(\nu_i)\,\dd x\\
     &> \int_{B^c} g(\bar{v}(x))\,\dd x
     + \int_B g(\bar{v}(x))\,\dd x + \varepsilon \lambda(B)\\
     &> \int_\Omega g(\bar{v}(x))\,\dd x .
\end{align*}
This concludes the proof.
\end{proof}

We note that the Lyapunov convexity theorem (see
\cite{lyapunov1940completely,tartar1979compensated})
asserts that such a sequence $(v^n)_n$ exists
for all functions $\bar{v}$ as in Proposition
\ref{prp:counter_statement_strict_convex}.

\subsection{Proof of \eqref{eq:diff_RRgn}}\label{sec:proof_diff_RRgn}

\begin{proof}
Let $w \in \mathcal{F}_{\eqref{eq:r}}$ and $\gamma > 0$
be given. Then, we obtain
\begin{gather}\label{eq:integral_identity}
R(w) - R_\gamma(w)
= \int_{\Omega} d(w(x))\,\dd x
\end{gather}
with
\[ d(u) \coloneqq g(u) - \inf\left\{ g(y) + \frac{1}{2\gamma}\|u - y\|^2_2\,\Big|\,y \in \conv\{\nu_1,\ldots,\nu_M\}\right\}.
\]
For $u \in \conv\{\nu_1,\ldots,\nu_M\}$, we consider $d(u)$ and rewrite it as
\[ d(u) = \sup\left\{
g(u) - g(y) - \frac{1}{2\gamma}\|u - y\|^2_2
\,\Big|\,y \in \conv\{\nu_1,\ldots,\nu_M\}\right\}.\]
Proposition \ref{prp:rel_mbreg_props} \ref{itm:glipschitz} gives
that $g$ is Lipschitz continuous with constant $L$, and we obtain
\[ d(u) \le \sup\left\{
L \|u - y\|_2 - \frac{1}{2\gamma}\|u - y\|^2_2
\,\Big|\,y \in \conv\{\nu_1,\ldots,\nu_M\}\right\}
\le \frac{L^2\gamma}{2},
\]
where the second inequality follows from
the maximization of a  parabola with negative curvature.
Inserting this estimate into \eqref{eq:integral_identity}
yields \eqref{eq:diff_RRgn}.
\end{proof}

\section{Alternative Smoothing in the One-Dimensional Case}\label{sec:1dim_subs}

Let $R$ be a relaxed multibang regularizer for discrete- and scalar-valued controls;
that is, $R(v) \coloneqq \int_{\Omega} g(v(x))\,\dd x$ for control functions $v \in \mathcal{F}_{\eqref{eq:r}}$.
Here, we consider the set of feasible controls for
\eqref{eq:r}:
\[ \mathcal{F}_{\eqref{eq:r}} =
\{ v \in L^2(\Omega) \,\vert\, v(x) \in [\nu_1,\nu_M] \text{ for a.a.\ } x \in \Omega \} \]
with scalars $\nu_1 < \ldots < \nu_M$.
We assume that $g : [\nu_1,\nu_M] \to \R$ is a positive,
continuous, montonously increasing,
piecewise affine convex function with $g(\nu_1) = 0$ and that
$g(\nu_M) = B \in \R$.

It is straightforward to generalize the following
ideas if the monotonicity assumption is dropped 
or $g(\nu_1)$ is allowed to be nonzero;
but this restriction simplifies the remainder significantly,
and we believe that it also helps to get a good intuition.

The Clarke subdifferential $\partial^c g$ of $g$ is
\[ \partial^c g(u)
   = \left\{
   \begin{aligned}
   \{ L_1 \} & \text{ if } u = \nu_1,\\
   \{ L_i \} & \text{ if } u \in (\nu_{i},\nu_{i+1}) \text{ for some } i \in [M-1],\\
   \{ L_{M-1} \} & \text{ if } u = \nu_M,\\
   [L_{i-1},L_{i}] & \text{ if } u = \nu_i \text{ for some } i \in [M-1]
   \end{aligned}
   \right.
   \]
for positive slopes $0 < L_1 < \ldots < L_{M-1}$.
We observe that $\partial^c g$ is
almost everywhere single-valued.
Thus, we may interpret $\partial^c g$ as an $L^\infty$-function,
which we denote by $g'$ because for this $L^\infty$-function
we still have
$g(u) = \int_{\nu_1}^u g'(w)\,\dd w$ by the fundamental theorem
of Lebesgue integral calculus. In other words, the function
$g$ is absolutely continuous.
For $v \in \mathcal{F}_{\eqref{eq:r}}$, we have $g(v(x)) \le B$ for a.a.\
$x \in \Omega$,
and thus $R(v) \le B \lambda(\Omega)$.
For $i \in [M]$,
we define $g_i \coloneqq g(\nu_i)$.

Now, we define differentiable convex underestimators of $g$.
We define the $C^1 \cap W^{2,\infty}$-function
$g_\gamma : [\nu_1,\nu_M] \to \R$
\[ g_\gamma(u) \coloneqq \int_{\nu_1}^u g'_\gamma(w)\,\dd w
\text{ for } u \in [\nu_1,\nu_M] \]
for $0 < \gamma < \min\{ \nu_{i+1} - \nu_{i} \,\vert\, i \in [M - 1]\}$,
where $g_\gamma'$ is defined as
\[  g_\gamma'(w) \coloneqq \left\{
  \begin{aligned}
  	L_1 & \text{ if } w \in [\nu_1,\nu_2),\\
  	L_{i-1} + \frac{L_{i} - L_{i-1}}{\gamma}(s - \nu_{i}) & \text{ if }
  	w \in [\nu_{i},\nu_{i} + \gamma)
  	\text{ for some } i \in [M - 1]\setminus \{1\},\\   
	L_i & \text{ if } w \in [\nu_{i} + \gamma, \nu_{i+1}) \text{ for some }
	i \in [M-1],\\ 	
    L_{M-1} & \text{ if } w = \nu_M
   \end{aligned}
   \right.
\]
for all $w \in [\nu_1,\nu_M]$. We show an example for
$g$ with underestimators and their derivatives
in Figure \ref{fig:lipschitz_underestimator_approx}.

\begin{figure}
\begin{subfigure}{.5\textwidth}
  \centering
  \includegraphics[width=\textwidth,height=4cm]{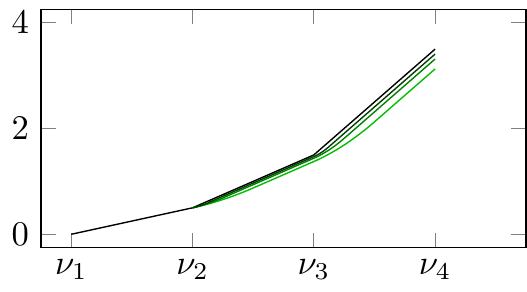}
%
%
%
%
  \caption{$g$ (black) and approximations $g_\gamma$.}
  \label{fig:sub1}
\end{subfigure}%
\begin{subfigure}{.5\textwidth}
  \centering
  \includegraphics[width=\textwidth,height=4cm]{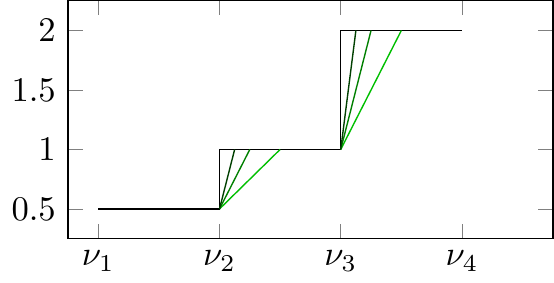}  
  \caption{$\partial^c g$ (black) and approximations $g_\gamma'$.}
  \label{fig:sub2}
\end{subfigure}
\caption{$g$ and $\partial^C g$ as well as their
         approximations $g_\gamma$ and $g_\gamma'$
         for the choices $\gamma \in \{2^{-1},2^{-2},2^{-3}\}$.}
         \label{fig:lipschitz_underestimator_approx}
\end{figure}

Following our notation, we define the
smoothed regularizer as
\begin{gather}\label{eq:1dim_moreau_substitute}
R_\gamma(v) \coloneqq \int_{\Omega} g_\gamma(v(x))\,\dd x
\end{gather}
for all $v \in \mathcal{F}_{\eqref{eq:r}}$.
We summarize the properties of the relatonship between
$g_\gamma$ and $g$ in the proposition below.

\begin{proposition}\label{prp:gdeltaEstimates}
Let $0 < \gamma < \min\{ \nu_{i+1} - \nu_{i} \,\vert\, i \in [M-1]\}$. It holds that
\begin{gather*}
\begin{aligned}
\left\|g - g_\gamma\right\|_{C([0,T])} &= 0.5(L_{M-1} - L_1)\gamma,\\
\left\|g' - g'_\gamma\right\|_{L^1([0,T])} &= 0.5(L_{M-1} - L_1)\gamma,\\
\left\|g' - g'_\gamma\right\|_{L^p([0,T])} &= \left(\sum_{i=1}^{M-2} (L_{i+1} - L_i)^p \frac{p}{p+1} \gamma  \right)^{\frac{1}{p}} \text{ for } p \in (1,\infty),\\
\|g' - g_\gamma'\|_{L^\infty([0,T])} &= \max\{L_{i+1} - L_i\,\vert\, i \in [M - 2]\}.	
\end{aligned}
\end{gather*}
\end{proposition}
\begin{proof}
By construction of $g_\gamma$, we have
$g(u) \ge g_{\gamma}(u)$ for all $u \in [\nu_1,\nu_M]$.
We consider the difference
$d_\gamma(u) \coloneqq g(u) - g_\gamma(u)$ for
$u$ in different intervals.
For $i \in [M]$,
let $d_\gamma^i \coloneqq g(\nu_i) - g_\gamma(\nu_i)$.

Because $g(u) = g_\gamma(u)$ for all $u \in [\nu_1,\nu_2]$
it holds that $d_\gamma^1 = d_\gamma^2 = 0$.
Let $i \in \{2,\ldots,M-1\}$.
For $u \in [\nu_i,\nu_{i+1}]$ it holds that
\begin{align*}
d_\gamma(u) &= d_\gamma^i 
	+ \int_{\nu_i}^{\min\{u, \nu_i + \gamma\}} 
      L_{i} - L_{i-1}
      - \left(\frac{L_{i} - L_{i-1}}{\gamma}\right)(w - \nu_i)
	\,\dd w\\
	&= d_\gamma^i + (L_{i} - L_{i-1})\min\{u - \nu_i,\gamma\}
	  - \frac{L_{i} - L_{i-1}}{2\gamma}
	    \min\{u - \nu_i,\gamma\}^2.
\end{align*}
Thus, $d_\gamma$ is a continuous montonously
nondecreasing function, and we obtain
$\sup_u d_\gamma(u) = d_\gamma(\nu_M)$.
Using the iterative description of $d_\gamma$ derived above
and the assumption
$\gamma < \min\{\nu_{i+1} - \nu_{i} \,\vert\, i\in [M-1]\}$,
we can compute
\begin{gather*}
d_\gamma(\nu_M) = \sum_{i=1}^{M-2} (L_{i+1} - L_i)\gamma - \frac{L_{i+1} - L_i}{2\gamma}\gamma^2 = \frac{L_{M-1} - L_1}{2}\gamma.
\end{gather*}

This implies $\|g - g_\gamma\|_C = 0.5(L_{M-1} - L_1)\gamma$.
Since $g' \ge g_{\gamma}'$ almost everywhere, it also holds that
$\|g' - g_\gamma'\|_{L^1} = 0.5(L_{M-1} - L_1)\gamma$.
Let $p \in (1,\infty)$. Then a reasoning similar to the above yields
\[ \|g' - g_{\gamma}'\|_{L^p([\nu_1,\nu_M])}
   = \left(\sum_{i=1}^{M-2} (L_{i+1} - L_i)^p \frac{p}{p+1} \gamma  \right)^{\frac{1}{p}}.
\]

By construction $\|g' - g'_\gamma\|_{L^\infty([\nu_1,\nu_{2}])} = 0$
and $\|g' - g'_\gamma\|_{L^\infty([\nu_i,\nu_{i+1}])} = L_{i} - L_{i-1}$
for all $i \in \{2,M - 1\}$, which yields the last claim.
\end{proof}

We  obtain the following corollary, which
establishes the properties of the Moreau envelope
from Proposition \ref{prp:moreau_gr}
for $R_\gamma$ and $g_\gamma$ as well.

\begin{corollary}
\begin{enumerate}
\item For all $u \in \conv V$ it holds that $g_\gamma(u) \uparrow g(u)$
for $\gamma \downarrow 0$.
\item For all $v \in \mathcal{F}_{\eqref{eq:r}}$ it holds that
$R_\gamma(v) \uparrow R(v)$ for $\gamma \downarrow 0$.
\item For all $\gamma > 0$, the function $g_\gamma : \conv V \to \R$
is convex and continuously differentiable with Lipschitz-continuous 
derivative.
\item Let $p \ge 2$. Then, $R_\gamma : L^p(\Omega,\R^m) \to \R$
is differentiable with derivative
$R_\gamma'(v)d  = \int_{\Omega} g_\gamma'(v(x))^T d(x)\,\dd x$.
\end{enumerate}
\end{corollary}
\begin{proof}
The claims follow along the lines of the proof of
Proposition \ref{prp:moreau_gr}.
\end{proof}

\begin{proposition}\label{prp:rdeltaGamma}
Let $(\gamma^n)_n \subset [0,\infty)$ satisfy
$\gamma^n \downarrow 0$.
Then, the sequence of functionals $F + R_{\gamma^n}$
on $\mathcal{F}_{\eqref{eq:r}} \subset L^\infty(\Omega,\R^m)$
is $\Gamma$-convergent with limit $F + R$
with respect to weak-$^*$-convergence in $L^\infty$.
\end{proposition}
\begin{proof}
The claim follows along the lines of the proof of
Proposition \ref{prp:rgammaGamma}. The estimates
on $g - g_\gamma$ can be taken directly from
Proposition \ref{prp:gdeltaEstimates}.
\end{proof}

\begin{theorem}\label{thm:1dim_moreau_subs_convergence_result}
Let $R$ be a relaxed multibang regularizer.
Let the inputs of Algorithm \ref{alg:sur_based_miocp_approximation}
be given. Let $\nu_1 < \ldots < \nu_M$.
Then, the assertions of Theorem \ref{thm:main_convergence_result}
hold true if the smoothing of $R$
is performed as defined in \eqref{eq:1dim_moreau_substitute}
instead of using the Moreau envelope.
\end{theorem}
\begin{proof}
This follows along the lines of
the proof of Theorem 
\ref{thm:main_convergence_result}
with the considerations above to replace
the estimates on the Moreau envelope.
\end{proof}

\bibliography{biblio}
\bibliographystyle{plain}

\section*{Government License}
~\\
\framebox{\parbox{.92\linewidth}{The submitted manuscript has been created by
UChicago Argonne, LLC, Operator of Argonne National Laboratory (``Argonne'').
Argonne, a U.S.\ Department of Energy Office of Science laboratory, is operated
under Contract No.\ DE-AC02-06CH11357.  The U.S.\ Government retains for itself,
and others acting on its behalf, a paid-up nonexclusive, irrevocable worldwide
license in said article to reproduce, prepare derivative works, distribute
copies to the public, and perform publicly and display publicly, by or on
behalf of the Government.  The Department of Energy will provide public access
to these results of federally sponsored research in accordance with the DOE
Public Access Plan \url{http://energy.gov/downloads/doe-public-access-plan}.}}

\end{document}